\documentclass
{amsproc}
\usepackage{amsmath}
\usepackage{amssymb}
\usepackage{graphicx}
\usepackage{multirow}
\usepackage{makecell}
\usepackage{array}
\DeclareGraphicsExtensions{.png,.pdf,.eps}
\usepackage[all,2cell]{xy}
\usepackage{tikz}
\usepackage{pgf}
\usepackage[customcolors]{hf-tikz}
\usepackage{enumerate}
\usepackage{mathdots}
\usepackage{enumitem}

\usetikzlibrary{cd}
\usepackage{tikz-cd}
\usepackage{hyperref}
\usepackage{ytableau}
\usepackage{pythonhighlight}

\usepackage{pifont} 
\newtheorem{theorem}{Theorem}[section]
\newtheorem{lemma}[theorem]{Lemma}
\newtheorem{corollary}[theorem]{Corollary}
\newtheorem{proposition}[theorem]{Proposition}

\theoremstyle{definition}

\newtheorem{definition}[theorem]{Definition}

\newtheorem{remark}[theorem]{Remark}

\newtheorem{set}[theorem]{Setup}

\newcolumntype{C}{>{$}c<{$}}

\renewcommand{\arraystretch}{1.2}

\title[Singular del Pezzo surfaces and isotropic flag varieties]{Singular del Pezzo surfaces and isotropic flag varieties} 

\author[Bianco]{Michele Bianco}
\address{Dipartimento di Matematica, Universit\`a degli Studi di Trento, via
Sommarive 14 I-38123, Trento (TN), Italy}
\email{michele.bianco@unitn.it}

\author[Sol\'a Conde]{Luis E. Sol\'a Conde}
\address{Dipartimento di Matematica, Universit\`a degli Studi di Trento, via Sommarive 14 I-38123, Trento (TN), Italy}
\email{eduardo.solaconde@unitn.it}

\thanks{}

\usepackage[textsize=tiny]{todonotes}

\usepackage{enumitem}

\newcommand\ignore[1]{}

\DeclareMathOperator{\HH}{H}
\def\conv{\operatorname{Conv}}

\def\C{{\mathbb C}}
\def\F{{\mathbb F}}
\def\G{{\mathbb G}}

\def\P{{\mathbb P}}
\def\Q{{\mathbb Q}}
\def\R{{\mathbb R}}
\def\Z{{\mathbb Z}}

\def\cC{{\mathcal C}}

\def\cG{{\mathcal{G}}}

\def\cI{{\mathcal I}}
\def\cJ{{\mathcal J}}

\def\cN{{\mathcal{N}}}
\def\cO{{\mathcal{O}}}

\def\cQ{{\mathcal{Q}}}

\def\cY{{\mathcal Y}}
\def\Q{{\mathbb{Q}}}
\def\G{{\mathbb{G}}}

\def\fn{{\mathfrak n}}

\def\fg{{\mathfrak g}}
\def\fh{{\mathfrak h}}

\def\fb{{\mathfrak b}}

\def\fsl{\mathfrak {sl}}

\def\fsp{\mathfrak {sp}}
\def\sl{\operatorname{SL}}

\def\operatorname#1{\mathop{\rm #1}\nolimits}

\def\DA{{\rm A}}

\def\DC{{\rm C}}
\def\DD{{\rm D}}

\def\Proj{\operatorname{Proj}}
\def\Aut{\operatorname{Aut}}

\def\Bl{\operatorname{Bl}}
\def\Bir{\operatorname{Bir}}

\def\Hom{\operatorname{Hom}}

\def\Pic{\operatorname{Pic}}
\def\Hom{\operatorname{Hom}}

\def\diag{\operatorname{diag}}

\def\deg{\operatorname{deg}}

\def\conj{\operatorname{conj}}

\def\NE{\operatorname{NE}}

\def\Nef{{\operatorname{Nef}}}

\def\Nu{{\operatorname{N_1}}}
\def\NU{{\operatorname{N^1}}}

\newcommand{\cNE}[1]{\overline{\NE}(#1)}

\def\PGL{\operatorname{PGL}}

\def\Sp{\operatorname{Sp}}
\def\PSP{\operatorname{PSp}}

\DeclareMathOperator{\No}{N}

\def\GZ{\mathcal{G}\! Z}
\def\GU{\mathcal{G} U}

\def\LZ{\mathcal{L} Z}

\def\ss{\operatorname{ss}}

\newcommand{\pb}{\ar@{}[dr]|{\text{\pigpenfont J}}}

\makeatletter
\newcommand{\xleftrightarrow}[2][]{\ext@arrow 3359\leftrightarrowfill@{#1}{#2}}
\makeatother
\newcommand{\xdasharrow}[2][->]{
\tikz[baseline=-\the\dimexpr\fontdimen22\textfont2\relax]{
\node[anchor=south,font=\scriptsize, inner ysep=1.5pt,outer xsep=2.2pt](x){#2};
\draw[shorten <=3.4pt,shorten >=3.4pt,dashed,#1](x.south west)--(x.south east);
}}

\newcommand\lra{\longrightarrow}

\def\Mo{\operatorname{\hspace{0cm}M}}

\begin{document}
\begin{abstract}
We compute the Chow quotient of the complete flag variety of isotropic subspaces of a four dimensional complex vector space with respect to a skew/symmetric form, and show that it is a singular del Pezzo surface of degree four. 
\end{abstract}
\maketitle

\tableofcontents


\section{Introduction}\label{sec:intro}

Geometric invariant theory and its variants provide powerful tools for constructing and studying moduli spaces arising as quotients of algebraic varieties by group actions. Among these constructions, the Chow quotient, introduced by Kapranov, occupies a distinguished place. Unlike GIT quotients, Chow quotients are canonical -- that is independent of a choice of linearization -- and retain all degenerations of general orbit closures. As a consequence, they typically exhibit more complicated singularities, while encoding finer information on boundary phenomena and birational geometry (see, for instance, \cite{Kap, KSZ,AB04})

In recent years, Chow quotients of flag varieties by maximal tori have attracted renewed attention, motivated both by their intrinsic geometry and by their connections with birational geometry, representation theory, and combinatorics. In our recent work 
\cite{BOS2}, we initiated a detailed study of these quotients in low rank, computing explicitly the Chow quotient of the complete flag variety of type $\DA_3$ (parametrizing complete flags in $\P^3$) and describing its birational geometry.

The purpose of the present paper is to extend this analysis to the symplectic setting, focusing on the complete flag variety $F$ of type $\DC_2$. Concretely, this $4$-dimensional variety parametrizes isotropic flags in $\C^4$ with respect to a fixed symplectic form. 
We study the (normalized) Chow quotient $X$ of $F$ by the action of a maximal torus $H$ of $\PSP(4)$. 

The algebraic cycle in $F$ corresponding to a general point in $X$ is a toric surface, which is the minimal resolution of a toric log del Pezzo surface of Picard number two and degree four, having four singularities of type $\DA_1$ (cf. Proposition \ref{prop:genorb}).  A key ingredient in our approach is a detailed analysis of local degenerations of the closure of a general torus orbit in the symplectic flag variety, in the spirit of the toric setting considered in \cite{KSZ,AB04}. By studying these degenerations explicitly, we gain precise control over the boundary of the Chow quotient, that is the locus parametrizing reducible cycles. One of our main results shows that this locus is a simple normal crossing divisor, given by the union of eight rational curves, whose incidence relations reflect the combinatorics of the torus action.

Moreover, we provide an explicit description of the general reducible cycle parametrized by each of these curves, identifying their irreducible components and intersection behavior inside the flag variety. Such reducible cycles naturally arise in Chow-type compactifications and are a primary source of singularities in the quotient, as already observed in related moduli-theoretic contexts (cf. \cite{Ale02}). This analysis plays a crucial role in determining the singularities of the Chow quotient and underlies our subsequent study of its birational geometry.\\

Using this description, our main result is the following. 

\begin{theorem}\label{thm:main}
Let $F$ be the complete flag variety of vector subspaces of $\C^4$, isotropic with respect to a nondegenerate skew-symmetric form, endowed with the natural action of the complex torus $H\subset\PSP(4)$ of homothety classes of diagonal matrices. Then the normalized Chow quotient $X$ of $F$ by the action of $H$ is a singular del Pezzo surface of degree $4$ with two singularities of type $A_1$. 
\end{theorem}
In particular, $X$ is a Mori Dream Space by \cite{BCHM}. We then study its nef and effective cones, and its anticanonical map, which is an embedding that maps $X$ to a quartic surface in $\P^4$. We also determine the automorphism group of the Chow quotient, showing that it coincides with the Weyl group of $\PSP(4)$.

This work underscores the usefulness of Chow quotients as a framework for understanding degenerations of orbits in flag varieties, from which Chow quotients of general rational homogeneous manifolds should follow. More generally, it supports the view that Chow quotients serve as natural ``master spaces'', dominating GIT quotients while retaining subtle boundary information relevant to birational geometry \cite{Kap,KSZ}.\\

\noindent{\bf Outline.}
We follow the lines of argumentation of \cite{BOS2}, adapted to our particular situation. After a section of preliminaries on torus actions and on the complete flag $F$ of type $\DC_2$,  we recall some known facts on its fixed points by the action of the maximal torus $H$, and introduce some particular $\C^*$-actions, that later on will play a central role in the study of the boundary divisors of the Chow quotient (cf. Section \ref{ssec:fixed}). Section \ref{sec:local} is devoted to the study of the combinatorial quotients of $H$-invariant affine charts in $F$, which contain information about the local degeneration of the closure of the general $H$-orbit in $F$. In Section \ref{sec:mainproof} we prove the main statement of the paper, showing how to construct the Chow quotient $X$ as an inverse limit of combinatorial quotients. We conclude the paper with a section devoted to a description of the geometry of $X$.

\noindent{\bf Acknowledgements.} We would like to thank Gianluca Occhetta for interesting discussions on the topic.

\section{Preliminaries}\label{sec:prelim}

\subsection{Torus actions and associated quotients}\label{ssec:prelimtorus}

In this section we quickly review some general result about torus actions on smooth projective varieties \cite{BWW} and their quotients \cite{KSZ,Kap,MFK}.
	
Consider an algebraic torus $H$ of rank $r$ effectively acting on a (complex) smooth projective variety $Z$. Denote by $\Mo(H)\simeq\Z^r$ its character lattice and by $\No(H)$ its dual lattice. Given an ample line bundle $L$ on $Z$, we get the {\it polarized pair} $(Z,L)$. We choose a {\it linearization} of $L$ by $H$, that is an $H$-action on $L$ that is linear along fibers and makes the projection $Z\to L$ an $H$-equivariant map. In this sense, the torus $H$ acts on the polarized pair $(Z,L)$. Linearizations always exist under these assumptions \cite[Proposition 2.4]{KKLV} and two linearizations of $L$ differ by a character of $H$ \cite[Lemma 2.2]{KKV}. 
	
\subsubsection{Polytopes of fixed points and sections.} In the above setting, the {\it fixed point locus} decomposes in smooth connected components $Z^H=\bigsqcup_{Y\in\cY} Y$ \cite[Theorem 1]{IVERSEN}. The induced $H$-action on the fiber $L_z$ of every $z\in Z^H$ naturally gives a {\it weight} $\mu_L(z)\in\Mo(H)$. Points in the same fixed point component $Y\in\cY$ give the same weight, thus we can define the character $\mu_L(Y):=\mu_L(z)$ for any $z\in Y$. Taking the convex hull of all the characters $\mu_L(Y), Y\in \cY$ in the real vector space $\Mo(H)_\R\simeq\R^r$, we get a convex polytope $\Delta(L)$, called the {\it polytope of fixed points} associated with the $H$-action on $(Z,L)$.
	
The same $H$-action provides another set of characters $\tilde \Gamma(L)\subset\Mo(H)$ corresponding to the decomposition $\HH^0(Z,L)=\bigoplus_{u\in \tilde\Gamma(L)}\HH^0(Z,L)_u,$ where $\HH^0(Z,L)_u$ is the $H$-eigenspace whose associated weight is $u$. The convex hull $\Gamma(L)\subset \Mo(H)_\R$ of $\tilde\Gamma(L)$ is called the {\it polytope of sections}. Remarkably, $\Delta(L)=\Gamma(L)$ holds if $L$ is globally generated \cite[Lemma 2.4]{BWW}. 

\subsubsection{GIT, Chow and limit quotients}
The polytope of sections $\Gamma(L)$ provides a set of GIT quotients of $Z$. In fact, every rational $u\in\Gamma(L)$ defines a finitely generated $\C$-algebra
\[A_u:=\bigoplus_{\substack{m\geq 0\\ mu\in\Mo(H)}} \HH^0(Z,mL)_{mu},\]
and the normal projective variety $\GZ_u:=\mathrm{Proj}(A_u)$ is a GIT quotient of $Z$ by the action of $H$. The map $Z\dashrightarrow \GZ_u$ is defined on a subset of semistable points $Z^{\ss}_u\subset Z$. We call these varieties the {\em GIT quotients of }$(Z,L)$. Every projective GIT quotient of $Z$ is a GIT quotient of a polarized pair $(Z,L)$, for some ample line bundle $L$. 

The following definition, given by Kapranov in \cite{Kap}, works in the broader context of linear group actions on projective varieties. We can always find an open subset $U\subset Z$ such that the closures $\overline{H.z}$ of the $H$-orbits of all points $z\in U$ have constant dimension and homology class. Up to shrinking $U$, we can assume that it is also $H$-invariant. Thus, we have a geometric quotient $\cG U:=U/H$ with a natural embedding in the {\it Chow variety} of algebraic cycles of $Z$ 
\[\psi:\cG U\hookrightarrow \mathrm{Chow}(Z).\]
\begin{definition}
    The {\it Chow quotient} of $Z$ by $H$ is the closure of $\psi(\cG U)$ in the Chow variety $\mathrm{Chow}(Z)$. We denote its normalization by $\cC Z$ and call it the {\it normalized Chow quotient} of $Z$ by $H$.
\end{definition}

It can be proven that Chow quotients by reductive group actions are related to GIT quotients \cite[Theorem 0.4.3]{Kap} but this relation is much tighter in the case of torus actions. Recall that the GIT quotients of $Z$ by the $H$-action form an inverse system. Indeed, if $Z^{\ss}_{u_2}\subset Z^{\ss}_{u_1}$ for $u_1\in \Gamma(L_1)$, $u_2\in \Gamma(L_2)$, there is a projective dominant morphism $f_{u_2,u_1}:\GZ_{u_2}\to \GZ_{u_1}$. 
Since the open set $U$ defined above is contained in every set of semistable points $Z^{\ss}_{u}$, we get a morphism $\GU\to \GZ_u$, for every $u$, and so a morphism from $\GU$ to the inverse limit of the GIT system. 
The unique irreducible component of the inverse limit of the GIT inverse system that contains the image of $\GU$ is called {\it limit quotient of $Z$} by $H$. Its normalization is denoted by $\LZ$ and is called the {\it normalized limit quotient}. 

Analogously, one may consider the inverse system of GIT quotients of $(Z,L)$, and define the {\it limit quotient of $(Z,L)$} by $H$, and its normalization $\LZ(L)$.

The following statement,  that directly links Chow quotient by torus actions to GIT quotients and their inverse limit,  has been proved in \cite{BHK}.


\begin{proposition}\cite[Corollary 2.6, 2.7]{BHK}\label{prp:BHK}
    Let $H$ be an algebraic torus effectively acting on a normal projective variety $Z$. Then $\cC Z\simeq\LZ\simeq \LZ(L)$, for every ample line bundle $L$ on $Z$. 
\end{proposition}

Let $\chi_u:\LZ\to\GZ_{u}$ be the natural projections for every $u\in\Gamma(L)$. Note that these $\chi_u$ are all birational morphisms. By definition, the normalized limit quotient satisfies the following universal property.

\begin{remark}\label{remark:univprop}
     Let $M$ be an irreducible normal variety with a collection of birational morphisms $\phi_u:M\to\GZ_u$ satisfying $\phi_u=f_{v,u}\circ\phi_v$ for every $u,v\in\Gamma$ such that $Z^{ss}_v\subset Z^{ss}_u$. Then there exists a unique birational morphism $\Phi:M\to\LZ$ such that $\phi_u=\chi_u\circ\Phi$ for all $u\in\Gamma$.
\end{remark}

\subsubsection{Combinatorial quotients} The definition of Chow quotients has a fully combinatorial description in the case of subtorus actions on toric varieties. This construction has been introduced in \cite{KSZ} with the name of {\it combinatorial quotient}. 

Now assume that $Z$ is a normal toric variety with maximal torus $T\subset Z$ and associated fan $\Sigma\subset\No(T)_\R$. Let $H\subset T$ be a subtorus and call $\pi:\No(T)_\R\to\No(T/H)_\R$ the natural projection. The {\it quotient fan} of $\Sigma$ by $H$ is the fan in $\No(T/H)_\R$ whose minimal cone containing $x\in\No(T/H)$ is given by \[\bigcap_{\substack{\sigma\in\Sigma\\x\in\pi(\sigma)}}\pi(\sigma).\]
\begin{definition}
    The normal $T/H$-toric variety whose fan is $\Sigma/H$ is called the {\it combinatorial quotient} of $Z$ by $H$.
\end{definition}

In the projective case, we have the following result.

\begin{proposition}\cite[Theorem 2.1]{KSZ}
    Let $Z$ be a projective $T$-toric variety and $H\subset T$ be a subtorus. The combinatorial quotient of $Z$ by $H$ is isomorphic to the normalized Chow quotient $\cC Z$.
\end{proposition}

We stress that combinatorial quotients can be performed for any toric variety, and they are still related to algebraic cycles of $Z$, although only topologically \cite[Proposition 2.1]{KSZ}.



\subsection{The complete flag of type $\DC_2$}\label{ssec:prelimFLAG}

In this section we introduce the complete flag variety $F$ of type $\DC_2$ and briefly describe its classical realization as the {\it null-correlation bundle} over $\P^3$.\\

Let $\fg=\fsp_4$ be the simple Lie algebra described by the Dynkin diagram $\DC_2$ 
and consider the associated adjoint linear group $G=\PSP(4)$, that is the \emph{projective symplectic group} of order $4$. By definition, $G$ is the quotient of $\Sp(4)$ by its center. Geometrically, $G$ is the group of automorphisms of $\P^3$ preserving a fixed non-degenerate skew-symmetric form $\Omega:\C^4\times\C^4\to\C$.
		

By a suitable choice of coordinates, we can assume that $\Omega$ is given by
\begin{equation}\label{eq:Omega}
    \Omega=\begin{pmatrix}
    0 & 0 & 0 & 1\\
    0 & 0 & 1 & 0\\
    0 & -1 & 0 & 0\\
    -1 & 0 & 0 & 0\\
\end{pmatrix}.\end{equation}
so that all elements of $\fg=\{g\in\fsl_4\,|\,g^t\Omega+\Omega g=0\}$ have the form
\begin{equation}\label{eq:sp4}
    g=\begin{pmatrix}
    h_1 & x'_1 & y' & z' \\
    x_1 & h_2 & x'_2 & y'\\
    y & x_2 & -h_2 & -x'_1\\
    z & y & -x_1 & -h_1
    \end{pmatrix}.
\end{equation}
In this way, elements of $G=\{M\in\sl(4)\,|\,M^t\Omega M=\Omega\}$ are homothety classes $[e^g]$ of exponential matrices taken all over $g\in\fg$.

Each one of the $8$ off-diagonal variables in (\ref{eq:sp4}) corresponds to one of the root spaces of $\fg=\fh\oplus\bigoplus_{\alpha\in \Phi}\fg_\alpha$, where $\fh$ is the Lie algebra of the rank-$2$ maximal torus 
\begin{equation}\label{eq:H}
    H=\{[\diag(h_1,h_2,h_2^{-1},h_1^{-1})]\in G\,|\, h_1,h_2\in\C^*\}
\end{equation} and $\Phi$ is the root system of $\fg$, represented in Figure \ref{fig:rootsC2}. Note that $\Mo(H)=\Z\Phi$. 

Let $\{e_1,e_2\}$ be the standard basis of $\R^2\simeq \Mo(H)_\R$ with the standard inner product. We choose $\alpha_1:=e_1$ as the simple short root while $\alpha_2:=e_2-e_1$ as the long simple one.

\begin{figure}[ht!]
\centering
\resizebox{0.3\textwidth}{!}{%
\begin{tikzpicture}
     \node at (0,0) {\includegraphics[width=3cm]{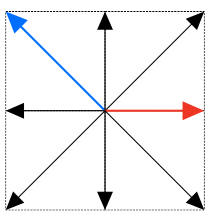}};
     \node at (1.8,0) {$\alpha_1$};
     \node at (-1.6,1.6) {$\alpha_2$};
     \node at (-0.1,1.6) {$\alpha_1+\alpha_2$};
     \node at (1.8,1.6) {$2\alpha_1+\alpha_2$};
\end{tikzpicture}
}
\caption{Root system of type $\DC_2$}\label{fig:rootsC2}
\end{figure}
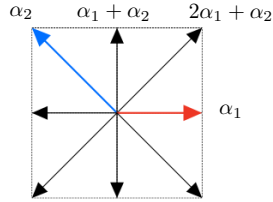

\begin{remark}
    The four negative roots in $\Phi$ corresponds to the four variables appearing in the strictly lower part of (\ref{eq:sp4}) in the following way:
    \begin{equation}\label{eq:root-var}
			x_1\longleftrightarrow-\alpha_1,\quad x_2\longleftrightarrow-\alpha_2,\quad y\longleftrightarrow-\alpha_1-\alpha_2,\quad z\longleftrightarrow-2\alpha_1-\alpha_2.
		\end{equation}
\end{remark}

The Weyl group $W$ of $G$ is isomorphic to the {dihedral group} $D_4$ of symmetries of a square, a Coxeter group of order $8$ with presentation
\begin{equation}\label{eq:presentation D4}
    D_4=\langle r_1,r_2\,|\, r_1^2,r_2^2,(r_1r_2)^4\rangle.
\end{equation}
By definition, elements of $W$ are classes of $H$-normalizing matrices in $G$ modulo $H$. We choose generators $r_1,r_2\in W$ associated to matrices \begin{equation}\label{eq:r1,r2}
    r_1:=\begin{pmatrix}
    0 & 1 & 0 & 0 \\
    1 & 0 & 0 & 0\\
     0 & 0 & 0 & 1\\
     0 & 0 & 1 & 0
    \end{pmatrix}\in\Sp(4),\qquad 
    r_2:=\begin{pmatrix}
        1 & 0 & 0 & 0\\
        0 & 0 & -1 & 0\\
        0 & 1 & 0 & 0\\
        0 & 0 & 0 & 1
    \end{pmatrix}\in\Sp(4).
\end{equation}
We keep denoting by $r_1,r_2$ the corresponding classes in $G$. The longest element $w_0\in W$ is $w_0=(r_1r_2)^2$ with length $4$. In this way, we can think the $8$ elements of $W$ as matrices \begin{equation}\label{eq:W}
    W=\{e:=\mathrm{id},r_1,r_2,r_1r_2,r_2r_1,r_1r_2r_1,r_2r_1r_2,w_0:=(r_1r_2)^2\}\subset\Sp(4).
\end{equation}


Then, let $B\subset G$ be the Borel subgroup of upper triangular matrices in $G$, which corresponds to the fixed choice of positive roots $\Phi^+$. In other terms, its Lie algebra is 
\begin{equation}\label{eq:b}
    \fb=\fh\oplus\bigoplus_{\alpha\in\Phi^+}\fg_\alpha.
\end{equation}

Since $H\subset B$, the Borel subgroup decomposes as $B=U\rtimes H$ with $U$ the unipotent subgroup of classes of strictly upper triangular matrices in $G$. In this way, we have the following description of the complete flag variety $F=G/B$: 
\begin{proposition}\label{prp:C2(1,2)}
    The complete flag variety $F$ is a rational smooth projective Fano 4-fold of Picard number $2$.
\end{proposition}

It is known that $F$ parametrizes $\Omega$-isotropic flags of vector subspaces in $\C^4$: $$\{(V_1,V_2)\in\G(1,4)\times\G(2,4)\,|\,V_1\subset V_2\,\,\, \Omega-\textrm{isotropic subspaces}\}$$ where $\G(k,4)$ denote the Grassmannian of $k$-vector subspaces of $\C^4$. Forgetting about planes, we get that $\P^3$ parametrizes isotropic lines in $\C^4$. On the other hand, the smooth quadric hypersurface $\cQ^3\subset\P^4$ parametrizes the isotropic planes of $\C^4$. This description provides two structures of $\P^1$-bundle for $F$, given by the natural {forgetful} maps
\begin{equation}\label{eq:C2(1,2) contractions}
    \begin{tikzcd}
        &F\dlar\drar\\
        \P^3 & &\cQ^3
    \end{tikzcd}
\end{equation}
which are exactly the only two possible contractions of $F$.

Another way to present $F$ is through the {\it null-correlation bundle} of $\P^3$. For details, see \cite[Section 4.2]{OSS}.

It is well-known that skew-symmetrical $2$-forms on $\C^4$ correspond to hyperplane sections of the smooth quadric hypersurface $\cQ^4\subset\P^5$, classically called {\it line complexes} since $\cQ^4\simeq\G(2,4)$. In particular, such forms are non-degenerate if and only if their corresponding hyperplane section of $\cQ^4$ is singular. We can use this last correspondence to construct a rank-$2$ vector bundle on $\P^3$ for any non-degenerate skew-symmetric form on $\C^4$. For example, consider the one denoted by $\Omega$ in (\ref{eq:Omega}). The polarity induced by $\Omega$ and the Euler sequence of $\P^3$ can be arranged together to give the following commutative diagram
\begin{equation*}
    \begin{tikzcd}
        \cO_{\P^3}\rar[hook]\dar[equal]&\Omega_{\P^3}(2)\rar[two heads]\dar[hook]&\cN_\Omega\dar[hook]\\
        \cO_{\P^3}\rar[hook, shorten >=-1.2em]\dar&\quad\,\,\C^4\otimes\cO_{\P^3}(1)\overset{\Omega}{\simeq}(\C^4)^\vee\otimes\cO_{\P^3}(1)\dar[two heads]\rar[two heads]& T_{\P^3}\dar[two heads]\\
        0 \rar &\cO_{\P^3}(2)\rar[equal]&\cO_{\P^3}(2)
    \end{tikzcd}
\end{equation*}
where rows and columns are exact and $\cN_\Omega$ is the rank-2 null-correlation bundle on $\P^3$ defined by $\Omega$. Moreover, taking global section we deduce the isomorphism \[\bigwedge^2(\C^4)^\vee\simeq\HH^0(\P^3,\Omega_{\P^3}(2))\simeq\Hom_{\cO_{\P^3}}(T_{\P^3},\cO_{\P^3}(2))\] which means that every non-degenerate skew-symmetric $2$-forms uniquely defines a null-correlation bundle on $\P^3$. 

Finally, consider the two natural contractions of $\P(T_{\P^3}^\vee)$:
\[
\begin{tikzcd}
    &\P(T_{\P^3}^\vee)\dlar["\pi_1"']\drar["\pi_2"]\\
    \P^3& &\cQ^4
\end{tikzcd}
\]

The form $\Omega$ determines an embedding $\cQ^3\hookrightarrow \cQ^4$, thus a subbundle $\pi_2^{-1}(\cQ^3)$ of $\P(T_{\P^3}^\vee)$. This is clearly isomorphic to the complete flag variety $F$ by our previous discussion about line complexes. Eventually, we see that $\pi_2^{-1}(\cQ^3)$ is also isomorphic to $\P(\cN_\Omega)$ as subbundles of $\P(T_{\P^3}^\vee)$ by construction of the null-correlation bundle. The two natural contractions of $\P(T_{\P^3}^\vee)$ can be lifted up to $\P(\cN_\Omega)$ so that we also naturally recover the two elementary contractions of $F$ presented in Diagram (\ref{eq:C2(1,2) contractions}).

\[\begin{tikzcd}
    &F\simeq\P(\cN_\Omega)\dlar["\pi_1"']\ar[dd,hook]\drar["\pi_2"]\\
    \P^3\ar[dd,equal] & & \cQ^3\ar[dd,hook]\\
    &\P(T_{\P^3}^\vee)\dlar["\pi_1"']\drar["\pi_2"]\\
    \P^3& &\cQ^4
\end{tikzcd}\]


\section{Maximal torus action on the complete flag}\label{sec:Chow}


\subsection{Fixed flags and their weights}\label{ssec:fixed}

In this section we describe the action on the complete flag variety $F$ of the rank two maximal torus $H\subset G$, given by diagonal matrices.

Note that every $h\in H$ uniquely corresponds to a couple $(h_1,h_2)\in(\C^*)^2$ as in (\ref{eq:H}). Thus our chosen positive roots can be described as follows:
\begin{equation}\label{eq:H-weights}
    \parbox{30mm}{$h^{\alpha_1}=h_1=:t_1,$\\ $h^{\alpha_2}=h_1^{-1}h_2=:t_2,$}\qquad \parbox{40mm}{$h^{\alpha_1+\alpha_2}=h_2=:t_1t_2,$\\ $h^{2\alpha_1+\alpha_2}=h_1h_2=:t_1^2t_2.$}
\end{equation}

We have the following description of the $H$-action on $F$. We denote by $\Lambda(H)$ the weight lattice of $H$ and recall that there is a natural isomorphism $\Lambda(H)\simeq\Pic(F)$ that associates to each weight $\lambda\in \Lambda(H)$ the line bundle \[L(\lambda):=G\times^B\C:=(G\times\C)/\sim, \quad (g,v)\sim(gb,\xi(b)^\lambda v)\quad\forall\,b\in B\]
where $\xi:B\to H$ sends an upper triangular matrix to its diagonal.
\begin{lemma}\cite[Propositions 3.6, 3.7]{WORS5}\label{lemma:fixedlocus}
    The fixed point locus $F^H$ is bijective to the Weyl group $W$, that is \[F^H=\{\sigma B\,|\,\sigma \in W\}.\] The weights of the induced $H$-action on the tangent space $T_{F,\sigma B}$ at each $\sigma B$ is $$\{\sigma(\alpha)\,|\,-\alpha\in\Phi^+\}.$$ 
    
    Moreover, for every weight $\lambda\in\Lambda(H)$, there exists an $H$-linearization $\mu^H_L$ of the associated line bundle $L=L(\lambda)\in\Pic(F)$ with \[\mu^H_L(\sigma B)=-\sigma(\lambda),\quad\forall\,\sigma\in W\] 
\end{lemma}

\begin{remark}
    In the following, we consider the complete flag variety $F$ polarized with the {\it minimal ample line bundle} $L=L(\lambda_{\mathrm{min}})$  defined by 
    \[\lambda_{\mathrm{min}}=\frac12(e_1+3e_2)\in\Lambda(H)\subset\Mo(H)_\R;\]  it has degree one with the fibers of $\pi_1$ and $\pi_2$ in Diagram (\ref{eq:C2(1,2) contractions}). The $H$-weights associated to the linearization of $L$ described in Lemma \ref{lemma:fixedlocus} are the vertices of their own convex hull in $\Mo(H)_\R$, which is the octagon $P$ in Figure \ref{fig:octagon}. In other terms, $P$ is the polytope of fixed points of the $H$-action on $F$.
\end{remark}
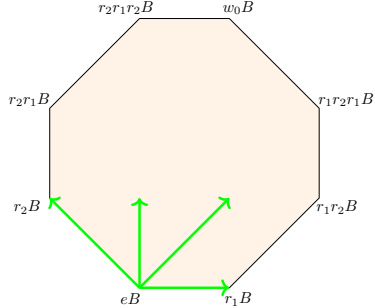
\begin{figure}[ht]
    \centering
    \resizebox{0.4\textwidth}{!}{%
\begin{tikzpicture}
    \node at (1.2,3.25) (nodee) {$w_0B$};
    \node at (-1.2,-3.25) (node1212) {$eB$};
    \node at (-1.3,3.25) (node1) {$r_2r_1r_2B$};
    \node at (1.2,-3.25) (node212) {$r_1B$};
    \node at (3.6,1.2) (node2) {$r_1r_2r_1B$};
    \node at (-3.5,-1.2) (node121) {$r_2B$};
    \node at (-3.45,1.2) (node21) {$r_2r_1B$};
    \node at (3.4,-1.2) (node12) {$r_1r_2B$};
    \fill[opacity=0.1,orange] (1,3) -- (3,1) -- (3,-1) -- (1,-3) -- (-1,-3) -- (-3,-1) -- (-3,1)-- (-1,3) -- (1,3);
    \draw (1,3) -- (3,1) -- (3,-1) -- (1,-3) -- (-1,-3) -- (-3,-1) -- (-3,1)-- (-1,3) -- (1,3);
    \draw[->, green, ultra thick] (-1,-3) -- (1,-3);
    \draw[->, green, ultra thick] (-1,-3) -- (1,-1);
    \draw[->, green, ultra thick] (-1,-3) -- (-1,-1);
    \draw[->, green, ultra thick] (-1,-3) -- (-3,-1);
\end{tikzpicture}
}\caption{\centering The octagon $P$ decorated with \\the positive roots centered at $eB$.}\label{fig:octagon}
\end{figure}


\subsection{Fundamental subtori and their associated BB-cells}\label{ssec:fundasubtori}

Since the character lattice $\Mo(H)$ is generated by the simple roots $\alpha_1,\alpha_2$, we can define the lattice homomorphisms $\mu_i:\Mo(H)\to\Z$, $i=1,2$ given by $\mu_i(\alpha_j)=\delta_{i,j}$. These homomorphisms provide $1$-parametric subgroups of $H$, which act on $F$. Following \cite{BWW}, the fixed point components of these two actions are precisely four $\P^1$'s in each case, each one containing two points of the form $\sigma B,\sigma\in W$. We have represented them as blue and red segments in Figure \ref{fig:2actions}. Note that for these two actions, the sink passes through $eB$ and the source passes through $w_0B$. 

\begin{figure}[ht]
    \centering
    \resizebox{0.25\textwidth}{!}{%
\begin{tikzpicture}
    \node at (1.2,3.25) (nodee) {$w_0B$};
    \node at (-1.2,-3.25) (node1212) {$eB$};
    \node at (-1.3,3.25) (node1) {$r_2r_1r_2B$};
    \node at (1.2,-3.25) (node212) {$r_1B$};
    \node at (3.6,1.2) (node2) {$r_1r_2r_1B$};
    \node at (-3.5,-1.2) (node121) {$r_2B$};
    \node at (-3.45,1.2) (node21) {$r_2r_1B$};
    \node at (3.4,-1.2) (node12) {$r_1r_2B$};
    \node[scale=1.4] at (0,-4.2) {Action $\mu_1$};
    \fill[opacity=0.1,orange] (1,3) -- (3,1) -- (3,-1) -- (1,-3) -- (-1,-3) -- (-3,-1) -- (-3,1)-- (-1,3) -- (1,3);
    \draw (1,3) -- (3,1) -- (3,-1) -- (1,-3) -- (-1,-3) -- (-3,-1) -- (-3,1)-- (-1,3) -- (1,3);
    \draw[blue, ultra thick] (1,3) -- (3,1);
    \draw[blue, ultra thick] (-1,3) -- (3,-1);
    \draw[blue, ultra thick] (1,-3) -- (-3,1);
    \draw[blue, ultra thick] (-1,-3) -- (-3,-1);  
\end{tikzpicture}
}
\hspace{2cm}
    \resizebox{0.25\textwidth}{!}{%
\begin{tikzpicture}
    \node at (1.2,3.25) (nodee) {$w_0B$};
    \node at (-1.2,-3.25) (node1212) {$eB$};
    \node at (-1.3,3.25) (node1) {$r_2r_1r_2B$};
    \node at (1.2,-3.25) (node212) {$r_1B$};
    \node at (3.6,1.2) (node2) {$r_1r_2r_1B$};
    \node at (-3.5,-1.2) (node121) {$r_2B$};
    \node at (-3.45,1.2) (node21) {$r_2r_1B$};
    \node at (3.4,-1.2) (node12) {$r_1r_2B$};
    \node[scale=1.4] at (0,-4.2) {Action $\mu_2$};
    \fill[opacity=0.1,orange] (1,3) -- (3,1) -- (3,-1) -- (1,-3) -- (-1,-3) -- (-3,-1) -- (-3,1)-- (-1,3) -- (1,3);
    \draw (1,3) -- (3,1) -- (3,-1) -- (1,-3) -- (-1,-3) -- (-3,-1) -- (-3,1)-- (-1,3) -- (1,3);
    \draw[red, ultra thick] (1,3) -- (-1,3);
    \draw[red, ultra thick] (3,1) -- (-3,1);
    \draw[red, ultra thick] (3,-1) -- (-3,-1);
    \draw[red, ultra thick] (1,-3) -- (-1,-3);
\end{tikzpicture}
}
    \caption{\centering Fixed point components of the $\C^*$-actions\\ associated to $\mu_1$ and $\mu_2$.}
    \label{fig:2actions}
\end{figure}
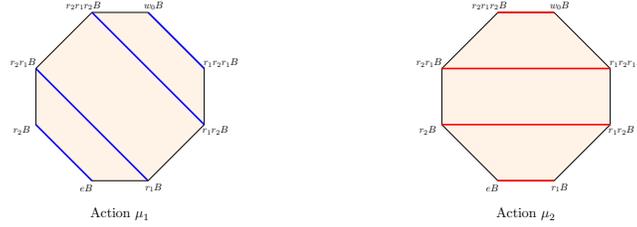

We can now give the following definition.

\begin{definition}\label{def:fundaction}
    The $\C^*$-actions on the flag manifold $F$ defined by the lattice homomorphisms of the form:
    \[\mu_i\circ\sigma:\Mo(H)\stackrel{\sigma}{\lra}\Mo(H)\stackrel{\mu_i}{\lra}\Z,\]
    for some $\sigma\in W$, will be called {\it fundamental $\C^*$-actions on $F$}.
\end{definition}

\begin{remark}\label{rem:fundaction}
 We will be interested in some BB-cells of the above actions; more concretely of the BB-cells associated to inner fixed components.
In this way, it will be enough for us to consider only $8$ out of the above $16$ actions, since the other $8$ can be obtained out of them by composition with the inversion $t\mapsto t^{-1}$. More precisely, we will denote $g:=r_2r_1=(r_1r_2)^3$, and consider:
\begin{enumerate}[leftmargin=6mm]
\item fundamental actions of the form $\mu_i\circ g^{k-1}$, $k=1,2,3,4$, $i=1,2$; 
\item for each of these actions $\mu_i\circ g^{k-1}$ we consider the BB-cells $F^\pm(Y_{i,k})$ of the fixed point component  $Y_{i,k}$ passing by the point $g^{k-1}r_iB$; in other words, this is the {\it inner} fixed point component of the action of minimal weight;
\item the subpolygons $P^\pm(i,k)$ generated by the weights of the $H$-fixed points contained in the closures of $F^\pm(Y_{i,k})$.
 \end{enumerate}
In this way, for each $i=1,2$, $k=1,2,3,4$, the pair $(i,k)$ provides a subdivision of the polygon $P$, that is \[P=P^-(i,k)\cup P^+(i,k).\] We will say that the subdivision is {\it of type $A$} if $i=2$ and {\it of type $B$} if $i=1$; furthermore, the subdivisions will be numbered as shown in Figure \ref{fig:subdiv}, as later explained in Section 6.
\end{remark}

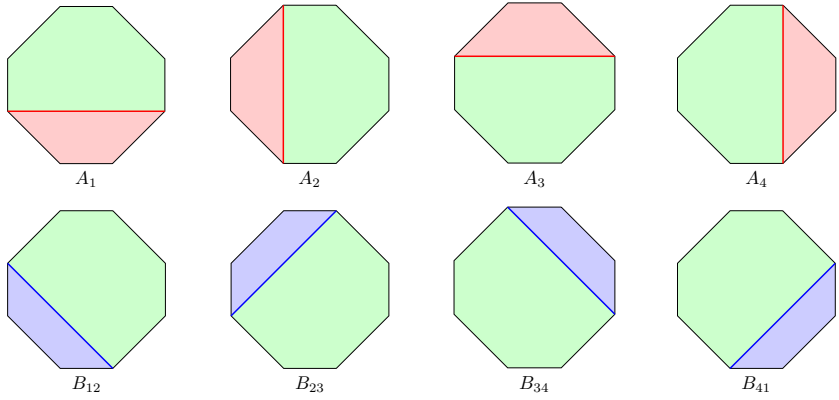
\begin{figure}[ht]
    \centering
    \resizebox{0.17\textwidth}{!}{%
\begin{tikzpicture}
    \node[scale=2] at (0,-3.6) {$A_1$};
    \fill[opacity=0.2,red] (1,-3) -- (3,-1) -- (-3,-1) -- (-1,-3) -- (1,-3);
    \fill[opacity=0.2,green]  (3,-1) -- (3,1) -- (1,3) -- (-1,3) -- (-3,1) -- (-3,-1)-- (3,-1);
    \draw (1,-3) -- (3,-1) -- (3,1) -- (1,3) -- (-1,3) -- (-3,1) -- (-3,-1)-- (-1,-3) -- (1,-3);
    \draw[red, ultra thick]  (3,-1) -- (-3,-1);
\end{tikzpicture}
}
\qquad
\resizebox{0.17\textwidth}{!}{%
\begin{tikzpicture}
    \node[scale=2] at (0,-3.6) {$A_2$};
    \fill[opacity=0.2,red] (-3,1) -- (-3,-1) -- (-1,-3) -- (-1,3) -- (-3,1);
    \fill[opacity=0.2,green]  (-1,-3) -- (1,-3) -- (3,-1) -- (3,1)-- (1,3) -- (-1,3) -- (-1,-3);
    \draw (1,3) -- (3,1) -- (3,-1) -- (1,-3) -- (-1,-3) -- (-3,-1) -- (-3,1)-- (-1,3) -- (1,3);
    \draw[red, ultra thick] (-1,3) -- (-1,-3);
\end{tikzpicture}
}
\qquad
\resizebox{0.17\textwidth}{!}{%
\begin{tikzpicture}
    \node[scale=2] at (0,-3.6) {$A_3$};
    \fill[opacity=0.2,red] (1,3) -- (3,1) -- (-3,1) -- (-1,3) -- (1,3);
    \fill[opacity=0.2,green] (3,1) -- (3,-1) -- (1,-3) -- (-1,-3) -- (-3,-1) -- (-3,1)-- (3,1);
    \draw (1,3) -- (3,1) -- (3,-1) -- (1,-3) -- (-1,-3) -- (-3,-1) -- (-3,1)-- (-1,3) -- (1,3);
    \draw[red, ultra thick] (3,1) -- (-3,1);
\end{tikzpicture}}
\qquad 
\resizebox{0.17\textwidth}{!}{%
\begin{tikzpicture}
    \node[scale=2] at (0,-3.6) {$A_4$};
    \fill[opacity=0.2,red] (3,1) -- (3,-1) -- (1,-3) -- (1,3) -- (3,1);
    \fill[opacity=0.2,green]  (1,-3) -- (-1,-3) -- (-3,-1) -- (-3,1)-- (-1,3) -- (1,3) -- (1,-3);
    \draw (1,3) -- (3,1) -- (3,-1) -- (1,-3) -- (-1,-3) -- (-3,-1) -- (-3,1)-- (-1,3) -- (1,3);
    \draw[red, ultra thick] (1,3) -- (1,-3);
\end{tikzpicture}
}
\\[4pt]
\resizebox{0.17\textwidth}{!}{%
\begin{tikzpicture}
    \node[scale=2] at (0,-3.6) {$B_{12}$};
    \fill[opacity=0.2,blue] (1,-3) -- (-1,-3) -- (-3,-1) -- (-3,1) -- (1,-3);
    \fill[opacity=0.2,green]  (1,-3) -- (-3,1) -- (-1,3) -- (1,3) -- (3,1) -- (3,-1) -- (1,-3);
    \draw (1,3) -- (3,1) -- (3,-1) -- (1,-3) -- (-1,-3) -- (-3,-1) -- (-3,1)-- (-1,3) -- (1,3);
    \draw[blue, ultra thick] (1,-3) -- (-3,1);
\end{tikzpicture}
}
\qquad
\resizebox{0.17\textwidth}{!}{%
\begin{tikzpicture}
    \node[scale=2] at (0,-3.6) {$B_{23}$};
    \fill[opacity=0.2,blue] (1,3) -- (-1,3) -- (-3,1) -- (-3,-1) -- (1,3);
    \fill[opacity=0.2,green]  (1,3) -- (-3,-1) -- (-1,-3) -- (1,-3) -- (3,-1) -- (3,1) -- (1,3);
    \draw (1,-3) -- (3,-1) -- (3,1) -- (1,3) -- (-1,3) -- (-3,1) -- (-3,-1)-- (-1,-3) -- (1,-3);
    \draw[blue, ultra thick] (1,3) -- (-3,-1);
\end{tikzpicture}
}
\qquad
\resizebox{0.17\textwidth}{!}{%
\begin{tikzpicture}
    \node[scale=2] at (0,-3.6) {$B_{34}$};
    \fill[opacity=0.2,blue] (-1,3) -- (1,3) -- (3,1) -- (3,-1) -- (-1,3);
    \fill[opacity=0.2,green] (-1,3) -- (3,-1) -- (1,-3) -- (-1,-3) -- (-3,-1) -- (-3,1) -- (-1,3);
    \draw (1,3) -- (3,1) -- (3,-1) -- (1,-3) -- (-1,-3) -- (-3,-1) -- (-3,1)-- (-1,3) -- (1,3);
    \draw[blue, ultra thick] (-1,3) -- (3,-1);
\end{tikzpicture}}
\qquad 
\resizebox{0.17\textwidth}{!}{%
\begin{tikzpicture}
    \node[scale=2] at (0,-3.6) {$B_{41}$};
    \fill[opacity=0.2,blue] (-1,-3) -- (1,-3) -- (3,-1) -- (3,1) -- (-1,-3);
    \fill[opacity=0.2,green]  (-1,-3) -- (3,1) -- (1,3) -- (-1,3) -- (-3,1) -- (-3,-1) -- (-1,-3);
    \draw (1,3) -- (3,1) -- (3,-1) -- (1,-3) -- (-1,-3) -- (-3,-1) -- (-3,1)-- (-1,3) -- (1,3);
    \draw[blue, ultra thick] (-1,-3) -- (3,1);
\end{tikzpicture}
}
    \caption{Subdivisions of type $A$, $B$.}
    \label{fig:subdiv}
\end{figure}


\section{Affine charts and their quotients}\label{sec:local}


We start by considering the nilpotent subalgebra 

\begin{equation}\label{eq:nilpotent}
    \fn:=\bigoplus_{\alpha\in \Phi^+}\fg_{-\beta}=\fg_{-\alpha_1}\oplus\fg_{-\alpha_2}\oplus\fg_{-\alpha_1-\alpha_2}\oplus\fg_{-2\alpha_1-\alpha_2}\subset \fg,
\end{equation}
that satisfies $\fg=\fb\oplus \fn$.

Composing the exponential map of $\fn$ with the quotient modulo $B$, we get a morphism, that is known to be an embedding of the affine space $\fn$ into the flag variety $F=G/B$ as a dense open subset $F_e\subset F$ 
\[\imath_e:\fn\stackrel{\simeq}{\lra} F_e:=\imath_e(\fn)\subset G/B, \qquad \imath_e(x)=\exp(x)B.\]
The map $\imath_e$ is $H$-equivariant, and so the open set $F_e$ is $H$-invariant. By composing $\imath_e$ with the elements of the Weyl group $W$, we get an $H$-invariant open covering of the flag variety $F$ given by
\[X=\bigcup_{\sigma\in W}F_\sigma,\qquad F_\sigma:=\sigma(F_e)=\sigma(\imath_e(\fn)).\]
\subsection{The general $H$-orbit} As in \cite[\S 3.2]{BOS2}, these affine charts $F_\sigma$ can be characterized in terms of $H$-orbits. The precise statement (which holds, with the same proof, for every complete flag variety) is the following.
\begin{proposition}\label{prop:Fsigma}
    For every $\sigma\in W$, it holds that \[F_\sigma=\{gB\in F\,|\,\sigma B\in\overline{H.gB}\}.\]
\end{proposition}
    \begin{proof}
        It suffices to show that the equality holds for $\sigma=e$. Consider a point $gB\in F$, so that we can write $g=\exp\left(\sum_{\alpha\in\Phi^+}g_{-\alpha}\right)$ for some $g_{-\alpha}\in\fg_{-\alpha}$. Then choose any $\mu\in\No(H)$ such that $\mu(\alpha)<0$ for all $\alpha\in\Phi^+$. It follows that \[t^\mu gB=\exp\left(\sum_{\alpha\in\Phi^+}t^{-\mu(\alpha)}g_{-\alpha}\right)B \,\,\,\mbox{for every}\,\, t\in\C^*,\] thus $eB=\lim_{t\to 0}t^\mu gB\in\overline{H.gB}$. Vice versa, since $F\setminus F_e$ is an $H$-invariant closed subset, if $gB\in F\setminus F_e$, we have $\overline{H.gB}\subseteq F\setminus F_e$ so, in particular, $eB\notin \overline{H.gB}$.
    \end{proof}

It is known \cite{CK00,Hu92} that, for  torus actions on projective varieties, the polytope of fixed points associated to a linearization of an ample line bundle coincides with the moment polytope of the toric variety obtained as the closure of the general orbit in the ambient projective variety. As a consequence of Proposition \ref{prop:Fsigma}, the closure of the orbit of a general point in a flag variety $F=G/B$ by the action of a maximal torus $H\subset G$ ($G$ semisimple) is a projective toric variety defined by the polytope of fixed points of the $H$-action on $(F,L)$, for any ample line bundle $L$ on $F$. In our symplectic setting, this general property can be stated as follows.

\begin{proposition}\label{prop:genorb}
    The closure of the general $H$-orbit in $F$ is the $H$-toric variety whose moment polytope is the octagon $P$. In particular, it is a toric degeneration of a del Pezzo surface of degree $4$. 
\end{proposition}

\begin{remark}\label{rem:genorb}
  The projective toric surface $S$ with moment polytope $P$ can be described as the blowup of a del Pezzo surface of degree six in two points, intersection of $(-1)$-curves, $P_1=E_1\cap E'_1$, $P_2=E_2\cap E'_2$, satisfying that $(E_1\cup E'_1)\cap (E_2\cup E'_2)=\emptyset$. The strict transforms in $S$ of those four $(-1)$-curves have self-intersection $-2$. The anticanonical linear system in $S$ is basepoint free, and the corresponding projective morphism is a contraction of the four $(-2)$-curves described above, onto a toric log del Pezzo surface of Picard number two and degree four, with four $\DA_1$ singularities (see \cite{KKN} for generalities on toric log del Pezzo surfaces).  
\end{remark}

\subsection{Combinatorial quotients of $H$-invariant affine charts} 

We will denote by $X_\sigma$ the combinatorial quotient of $F_\sigma$ by the action of $H$, for every $\sigma\in W$. 

\begin{remark}\label{rem:Xsigma}
    Given $h\in H$, we have a commutative diagram
    \[
    \begin{tikzcd}[column sep = large]
            F_e\rar["h"]\dar["\sigma"']&F_e\dar["\sigma"]\\
            F_\sigma\rar["\conj_\sigma(h)"']& F_\sigma
    \end{tikzcd}
    \]
    So, although $\sigma:F_e\to F_\sigma$ is not $H$-equivariant, we still have that $\sigma$ induces an isomorphism among the combinatorial quotients of $F_e$ and $F_\sigma$, denoted by
    \[
    q_\sigma:X_e\to X_\sigma.
    \]
    On the other hand, we have a $H$-equivariant birational map provided by inclusions $F_e,F_\sigma\subset F$ which then induces a birational map among the corresponding combinatorial quotients, denoted by 
    \[
    \jmath_\sigma:X_e\dashrightarrow X_\sigma.
    \]
\end{remark}

\begin{definition}\label{def:bsigma}
For every $\sigma \in W$ we denote by $b_\sigma\in \Bir(X_e)$ the composition
\[
b_\sigma:=\jmath_{\sigma^{-1}}\circ q_{\sigma}:X_e \lra X_\sigma \dashrightarrow X_e.
\]
The above correspondence obviously defines a group homomorphism, denoted by \[b:W\to \Bir(X_e).\]
\end{definition}

\begin{proposition}\label{prop:combiquot}
The combinatorial quotients $X_\sigma$ are projective surfaces. 
\end{proposition}

\begin{proof}
By Remark \ref{rem:Xsigma}, we have that $X_\sigma\simeq X_e$, so it is enough to consider the case $\sigma=e$.

Let us consider the $H$-eigenspace decomposition of $\fn$ (see Equation (\ref{eq:nilpotent})), and denote by $T\subset \fn$ the $4$-dimensional torus defined as the set of elements whose projections to the $H$-eigenspaces  are different from zero. By choosing a basis of $H$-eigenvectors of $\fn$ we get induced isomorphisms $T\simeq (\C^*)^4$, $\No(T)\simeq \Z^4$ such that, by identifying $\No(H)=\Z(\alpha_1,\alpha_2)^\vee$ with $\Z^2$, the action of $H$ on $\fn$ is determined by the lattice monomorphism $\No(H)\to \No(T)$ given by the matrix:
\[
-\begin{pmatrix}1&0\\0&1\\1&1\\2&1\end{pmatrix}.
\]
The quotient map $\pi:\No(T)\to \No(T/H)$ can be written as:
\[\xymatrix@R=-2mm@C=3mm{&{\begin{pmatrix}-1&-1&1&0\\-1&0&-1&1\end{pmatrix}}&\\
\No(T)\ar[rr]&&\No(T/H)
}\]
One then easily gets that the quotient fan defining $X_e$ is the complete fan $\Sigma$ having the column vectors of the above matrix as ray generators. We have represented it in Figure \ref{fig:quotfan}.
\begin{figure}[ht!!]
    \centering
    \includegraphics[width=0.2\linewidth]{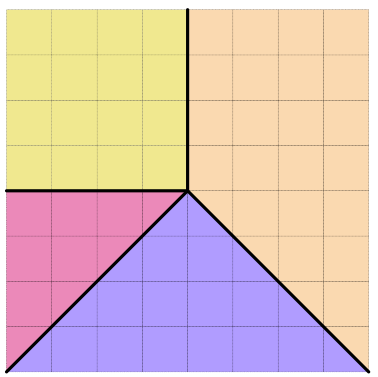}
    \caption{The quotient fan of $\fn$ by the action of $H$}
    \label{fig:quotfan}
\end{figure}
In particular, $X_e$ is a complete normal toric surface, hence projective.
\end{proof}

\begin{remark}\label{rem:combiquot}
    Note that the quotient fan computed above allows us to describe $X_e$ in the following two ways:
    \begin{itemize}[leftmargin=6mm]
        \item $X_e$ is the blowup of a quadric cone surface $Q^2$ on a smooth point;
        \item $X_e$ is a weighted blowup of $\P^2$.
    \end{itemize}

In other words, we have two contractions of $X_e$, one to $Q^2$, and one to $\P^2$, that are represented by the following inclusions of moment polytopes:

\begin{figure}[h!]
\centering
\resizebox{0.6\textwidth}{!}{%
\begin{tikzpicture}
     \node at (0,0) {\includegraphics[width=0.5\textwidth]{ 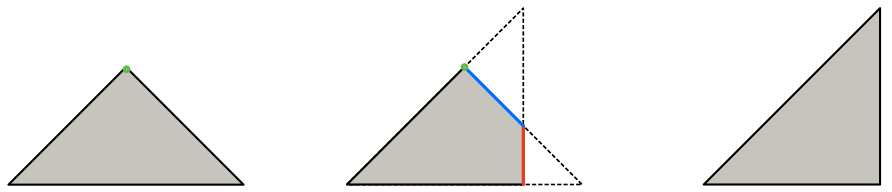}};
     \draw[->] (-0.8,-0.2) -- (-1.6,-0.2);
     \draw[->] (1.1,-0.2) -- (1.9,-0.2);
     \node[scale=0.8] at (-2.5,-1) {$Q^2$};
     \node[scale=0.8] at (0,-1) {$X_e$};
     \node[scale=0.8] at (2.5,-1) {$\P^2$};
     \draw[->] (-0.3,-1) -- (-2.2,-1);
     \draw[->] (0.3,-1) -- (2.2,-1);
     \node[scale=0.8] at (1.25,-1.2) {$p$};
     \node[scale=0.8] at (-1.25,-1.2) {$p'$};
     \node[scale=0.7] at (0.65,-0.5) {\color{red}$A$};
     \node[scale=0.7] at (0.4,0.15) {\color{blue}$B$};
     \end{tikzpicture}
}
\end{figure}

We will denote these maps by $p:X_e\to \P^2$, $p':X_e\to Q^2$, and their exceptional divisors by $A,B$, respectively. Note that $X_e$ has a singular point on the exceptional divisor $B$ (represented as a green point in the above figure).
\end{remark}


\begin{lemma}\label{lem:twoGIT}
With the same notation as above, the images of the contractions $p,p'$ are geometric GIT quotients of $F_e$, and $X_e$ is a minimal resolution of indeterminacies of the induced rational map among these two quotients.
\end{lemma}

\begin{proof}
Let us denote by  $\{w_1,w_2,w_3,w_4\}$ the $\Z$-basis of $\No(T)$ induced by the choice of a basis of $H$-eigenvectors of $\fn$ as in the proof of Proposition \ref{prop:combiquot}. The fan $\Sigma(\tau)$ of $F_e$ as a $T$-toric variety is the fan of faces of the positive orthant $\tau$, generated by the $w_i$'s.
For every subset $I\subset \{1,2,3,4\}$, we will denote by $\tau_I$ the cone generated by $w_i$, $i\in I$, and set $\tau_\emptyset=\{0\}$.

Let us consider the following subfans of $\Sigma(\tau)$:
\begin{itemize}
    \item $\Sigma(\tau)_p\hspace{0.1cm}:=\big\{\tau_I\,\,|\,\,I=\emptyset,\{2\},\{3\},\{4\},\{2,3\},\{2,4\},\{3,4\}\big\},$
    \item $\Sigma(\tau)_{p'}:=\big\{\tau_I\,\,|\,\,I=\emptyset,\{1\},\{3\},\{4\},\{1,3\},\{1,4\},\{3,4\}\big\}.$
\end{itemize}
The inclusion of these subfans into $\Sigma(\tau)$ defines two open subsets  $U_p:=X(\Sigma(\tau)_p)$, $U_{p'}:=X(\Sigma(\tau)_{p'})$ of $F_e=X(\Sigma(\tau))$. Moreover, these two subfans satisfy that their cones project via $\No(T)\to \No(T/H)$ isomorphically to the cones of two fans $\Sigma_p, \Sigma_{p'}$; in other words, all the $H$-orbits on the open sets $U_p,U_{p'}$ are two dimensional,  and the induced maps
\[U_p\lra X(\Sigma_p),\quad U_{p'}\lra X(\Sigma_{p'})\]
are geometric quotients. 
Finally, the fan $\Sigma$ that defines $X_e$ is the coarsest common refinement of $\Sigma_p,\Sigma_{p'}$, so we get a minimal resolution of indeterminacies. 
\[
\begin{tikzcd}
    &X_e=X(\Sigma)\ar[dr,"p"]\ar[dl,"p'"']\\
    X(\Sigma_{p'})=Q^2\ar[rr, dashed]& &\P^2=X(\Sigma_p)
\end{tikzcd}
\] In conclusion, following \cite[Lemma 2.4]{BOS2}, we have that $X(\Sigma_p), X(\Sigma_{p'})$ are GIT quotients of $F$ by suitable linearizations of the $H$-action.
\end{proof}

\begin{corollary}\label{cor:combGITlimit}
    The combinatorial quotient $X_e$ is an inverse limit of GIT quotients of $F$. In particular, there exists an induced morphism $\pi_e:X\to X_e$.
\end{corollary}
\begin{proof}
   Setting:  \[\Sigma(\tau)^\circ:=\big\{\tau_I\,\,|\,\,I=\emptyset,\{1\},\{2\},\{3\},\{4\},\{1,2\},\{2,4\},\{3,4\},\{1,3\}\big\}\subset\Sigma(\tau),\]
    one may check that the projection $\No(T)\to \No(T/H)$ sends $\Sigma(\tau)^\circ$ to $\Sigma$, preserving the dimension of the cones. As in Lemma \ref{lem:twoGIT}, we may then assert that $X_e$ is a geometric GIT quotient of $F$ by the action of $H$. In particular, $X_e$ is the inverse limit of three geometric GIT quotients of $F$ hence, by Proposition \ref{prp:BHK}, there exists an induced morphism from the limit quotient $X$ to $X_e$. 
\end{proof}
\begin{remark}\label{remark:combGITlimit}
    Since the induced map $\pi_e:X\to X_e$ is birational by construction, and $X_e$ is normal, then it is a contraction of $X$. In the same way, for every $\sigma\in W$ we get a contraction
    \[\pi_\sigma:X\lra X_\sigma.\]
    Note that, by construction, $\pi_\sigma\circ\pi_e^{-1}=\jmath_\sigma$, for every $\sigma\in W$.

    Now denote by $\tilde X$ the normalization of the image of the product morphism
    \begin{equation*}
        (\pi_\sigma)_{\sigma\in W}: X\longrightarrow\prod_{\sigma\in W} X_\sigma
    \end{equation*}
    and let $\tilde \pi: X\to\tilde X$ be the induced morphism.
\end{remark}

\begin{proposition}\label{prp:invlim}
    The map $\tilde\pi:X\to \tilde X$ is an isomorphism.
    \begin{proof}
        By construction, $\tilde\pi$ is a birational map between two normal varieties, thus it suffices to prove that it is finite. Assume by contradiction that $\tilde\pi$ contracts an irreducible curve $C\subset X$. Then every $\pi_\sigma$ contracts $C$, too, and so every $H$-invariant cycle parametrized by $C$ has the same fixed component in each $F_\sigma$. This contradicts that $X=\bigcup_{\sigma\in W} F_\sigma$.
    \end{proof}
\end{proposition}

\begin{remark}\label{rem:invlim}
    By Proposition \ref{prp:invlim}, the Chow quotient $X$ satisfies the following universal property: if $Y$ is a normal irreducible variety with birational morphisms $\varphi_\sigma: Y\to X_\sigma$ such that $\varphi_\sigma=\jmath_\sigma\circ\varphi_e$ for every $\sigma\in W$, then there exists a unique (birational) morphism $\varphi: Y\to X$ satisfying $\varphi_\sigma=\pi_\sigma\circ\varphi$ for all $\sigma\in W$. In this sense, we say that $X$ is the {\it inverse limit} of the combinatorial quotients $X_\sigma$.
\end{remark}

The following observation on the geometric quotients $X(\Sigma_p),X(\Sigma_{p'})$ will be useful later on.

\begin{remark}\label{rem:geomP2}
    The subsets $U_p,U_{p'}$ can be written as:
    \[\begin{array}{l}
U_p=(\fg_{-\alpha_1}\setminus\{0\})\oplus\big(\fg_{-\alpha_2}\oplus\fg_{-\alpha_1-\alpha_2}\oplus\fg_{-2\alpha_1-\alpha_2}\setminus\{(0,0,0)\}\big)\subset\fn,\\
U_{p'}=(\fg_{-\alpha_2}\setminus\{0\})\oplus\big(\fg_{-\alpha_1}\oplus\fg_{-\alpha_1-\alpha_2}\oplus\fg_{-2\alpha_1-\alpha_2}\setminus\{(0,0,0)\}\big)\subset\fn.
\end{array}
    \]
    Note then that the quotients of these open sets by $H$ can be described as certain quotients of the fundamental $\C^*$-actions, namely 
    \[
(\fg_{-\alpha_2}\oplus\fg_{-\alpha_1-\alpha_2}\oplus\fg_{-2\alpha_1-\alpha_2}\setminus\{(0,0,0)\})/\mu_2,\]\[(\fg_{-\alpha_1}\oplus\fg_{-\alpha_1-\alpha_2}\oplus\fg_{-2\alpha_1-\alpha_2}\setminus\{(0,0,0)\})/\mu_1.
    \]
    These are the standard projective space $\P^2$, and the weighted projective space $\P(1,1,2)\simeq Q^2$, respectively. In particular, the image into $\P^2$ corresponding to the closure of orbit of a general element $(x_1,x_2,y,z)\in \fn$ is the point $(x_2:x_1^{-1}y:x_1^{-2}z)$.
\end{remark}

\subsection{The quotient fan and the BB-decomposition}\label{ssec:combiquotBB}

Let us now provide an interpretation of the exceptional divisors $A,B\in X_e$ introduced previously. As we shall see, they are related to the polygon subdivisions shown in Remark \ref{rem:fundaction}.

For brevity, we will consider only the case of the divisor $A$; the case of the divisor $B$ is analogous. We have represented in Figure \ref{fig:combiquotBBA} the cones of the fan of $\Sigma$ involved in the contraction $p'$, and we will consider their counterparts in the fan $\Sigma(\tau)$, which are $\langle w_1,w_4\rangle, \langle w_2\rangle $. 

Consider now the subdivision of the octagon \[P=P^-(2,2)\cup P^+(2,2),\] that we have denoted by $A_2$ in Remark \ref{rem:fundaction}. One may immediately check that the corresponding BB-cells $F^\pm(Y_{2,2})$ satisfy the following property
\[\begin{cases}T_{F^-(Y_{2,2}),eB}=\fg_{-\alpha_2}\oplus\fg_{-\alpha_1-\alpha_2}\\T_{F^+(Y_{2,2}),eB}=\fg_{-\alpha_1-\alpha_2}\oplus \fg_{-2\alpha_1-\alpha_2}\oplus\fg_{-\alpha_1}.\end{cases}\]

\begin{figure}[ht]
    \centering    
    \resizebox{0.2\textwidth}{!}{%
\begin{tikzpicture}
    \node[scale=2] at (0,-3.6) {$A_2$};
    \fill[opacity=0.2,red] (-3,1) -- (-3,-1) -- (-1,-3) -- (-1,3) -- (-3,1);
    \fill[opacity=0.2,green]  (-1,-3) -- (1,-3) -- (3,-1) -- (3,1)-- (1,3) -- (-1,3) -- (-1,-3);
    \draw (1,3) -- (3,1) -- (3,-1) -- (1,-3) -- (-1,-3) -- (-3,-1) -- (-3,1)-- (-1,3) -- (1,3);
    \draw[red, ultra thick] (-1,-3) -- (-1,3);
    \draw[->, ultra thick, green](-1,-3) -- (1,-3);
    \draw[->, ultra thick, green](-1,-3) -- (1,-1);
    \draw[->, ultra thick, green](-1,-3) -- (-1,-1);
    \draw[->, ultra thick, green](-1,-3) -- (-3,-1);
\end{tikzpicture}
}
\hspace{3cm}
    \resizebox{0.25\textwidth}{!}{%
\begin{tikzpicture}
     \node at (0,0) {\includegraphics[width=0.25\textwidth]{ 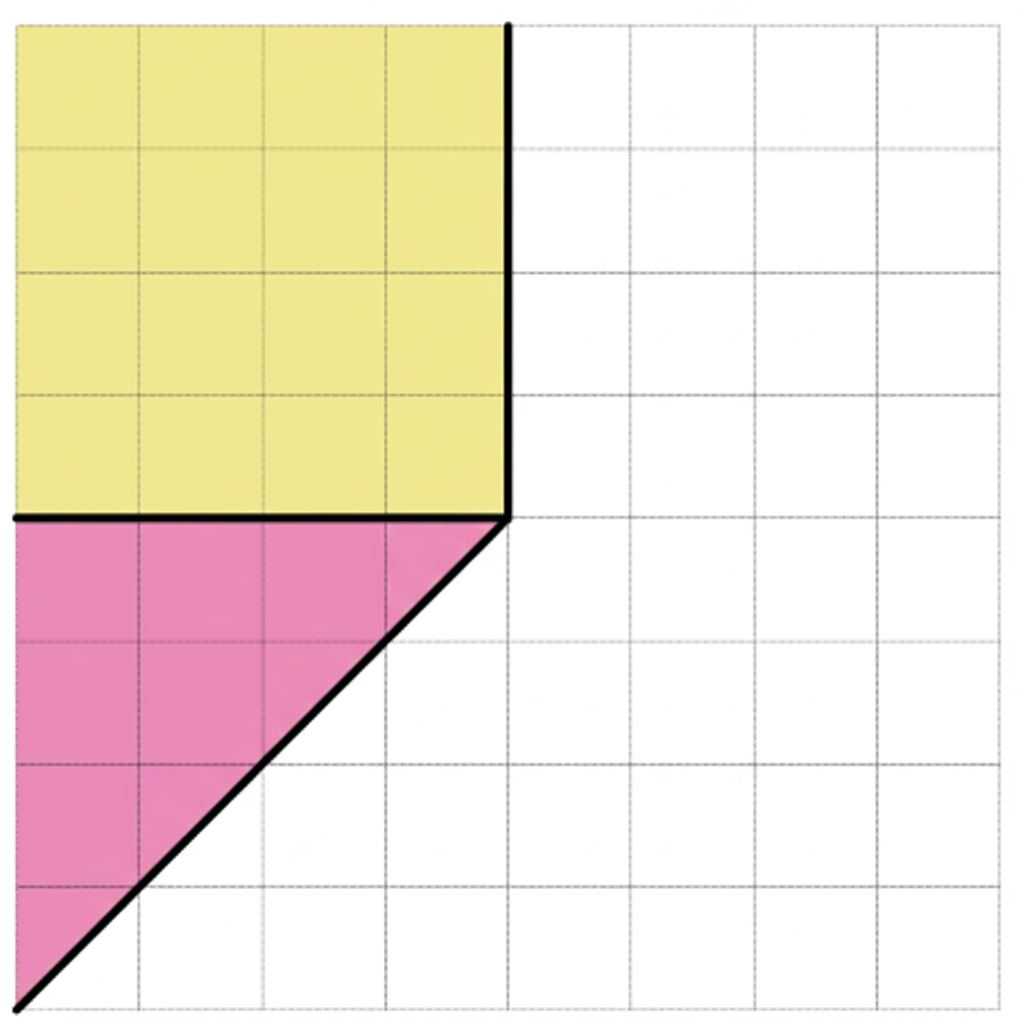}};
     
     \node[scale=0.7] at (0.4,0.65) {$\pi(w_4)$};
     \node[scale=0.7] at (-0.6,0.2) {$\pi(w_2)$};
     \node[scale=0.7] at (-0.45,-0.9) {$\pi(w_1)$};
     \end{tikzpicture}
}
    \caption{\centering The subdivision $A_2$ and a local toric contraction \\of the corresponding contraction.}
    \label{fig:combiquotBBA}
\end{figure}
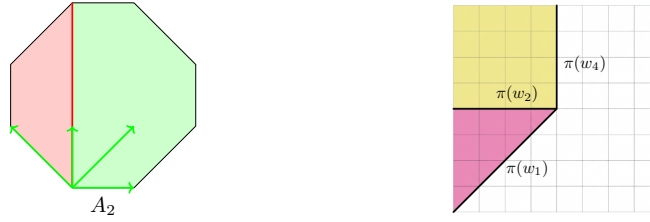

In other words,  $F^\pm(Y_{2,2})\cap F_e$ are the closures of the $T$-orbits associated to the faces of the orthant $\tau$:
\[\rho^-=\langle w_1,w_4\rangle,\,\, \rho^+=\langle w_2\rangle,\] which turn out to be the cones involved in the contraction $p'$.

     

The interpretation of the above fact is the following. Consider the $T/H$-orbit in $X_e$  associated to the cone $\pi(\rho^+)=\langle\pi(w_2)\rangle$; 
it parametrizes a $1$-dimensional family of $H$-invariant subvarieties, such that:
\begin{itemize}[leftmargin=6mm]
\item in the open set $U_p$ it consists of the $1$-dimensional family of $H$-orbits contained in $O(\rho^+)$;
\item in the open set $U_{p'}$ it consists of a single $H$-orbit, namely $O(\rho^-)$;
\item the closure of $O(\rho^-)$ and the closures of the $H$-orbits in $O(\rho^+)$ contain the orbit $O(\langle w_1,w_2,w_4\rangle )$, which consists of a single $1$-dimensional $H$-orbit.
\end{itemize}
In other words, the curve $O(\pi(\rho^+))\subset X_e$ may be interpreted as a family of reducible $H$-invariant subvarieties, each one consisting of an $H$-orbit in $O(\rho^-)$, the $H$-orbit $O(\rho^+)$, and their common boundary $O(\langle w_1,w_2,w_4\rangle )$. 

\subsection{Boundary divisors} 
    In Section \ref{ssec:combiquotBB} we interpreted the exceptional divisors $A,B$ of the combinatorial quotient $X_e$ in relation with certain subdivisions of the octagon $P$. Repeating the quotient construction for every $F_\sigma$, we get a pair of analogous divisors of type $A,B$ in every combinatorial quotient $X_\sigma$. Such divisors correspond to families of reducible $H$-invariant cycles and their contractions onto the corresponding geometric GIT quotients describe their components in $F_\sigma$.

\begin{definition}
    We call {\it boundary divisors} (or {\it boundary curves}) of $X$ the strict transforms via $\pi_\sigma:X\to X_\sigma$ of the exceptional divisors of type $A,B$ in $X_\sigma$, for every $\sigma\in W$.
\end{definition}

\begin{remark}
    By Proposition \ref{prp:invlim}, boundary divisors of $X$ parametrizes all the reducible $H$-invariant cycles of $F$.
\end{remark}

\begin{remark}
In virtue of their local nature, exceptional divisors in different $X_\sigma$'s may give rise to coincident boundary divisors in $X$. The following result tells us that  boundary divisors correspond to the $8$ subdivisions  of the octagon described above  ($A_1,A_2,A_3,A_4$, and $B_{12},B_{23},B_{34},B_{41}$). In particular, this will allow us to denote boundary divisors by their unique corresponding subdivision of $P$. 
\end{remark}

\begin{proposition}\label{prop:8bounddiv}
    There are $8$ distinct boundary divisors in $X$ corresponding to the subdivisions of type $A,B$ of the octagon $P$.
    \begin{proof}
        The exceptional divisor of type $A$ (resp. $B$) in $X_\sigma$ corresponds to the unique subdivision of type $A$ (resp. $B$) of $P$ associated to the fixed point component through $\sigma B\in F^H$ of some fundamental $\C^*$-action. Therefore, the strict transform in $X$ of two exceptional divisors from different combinatorial quotients coincide if and only if the components of their parametrized $H$-invariant reducible cycles agree along the same fixed point component. Equivalently, exceptional divisors in different combinatorial quotients get transformed into the same boundary divisor of $X$ if and only if they correspond to the same subdivision of $P$. 
    \end{proof}
\end{proposition}

\begin{remark}
    The subdivisions of the octagon $P$ also provide a description of the reducible cycles parametrized by the general element in the corresponding boundary divisor in $X$; they will have two components $\overline{F^\pm(Y_{i,k})}$, which are $H$-toric projective varieties whose moment polytopes are the subpolytopes $P^\pm(i,k)\subset P$, respectively.
    By means of the $W$-action on $X$, cycles parametrized by boundary divisors of the same type ($A$ or $B$) have isomorphic components, so it is enough to describe $A_2$ and $B_{41}$. Below we will denote by $S$ the closure of the general orbit in $F$ (see Remark \ref{rem:genorb}).
    \begin{itemize}[leftmargin=6mm]
        \item[($A_2$)] 
        In this case 
        $\overline{F^-(Y_{2,2})}$ is a Hirzebruch surface $\F_2$, and $\overline{F^+(Y_{2,2})}$ is the blowup of $\P^1\times\P^1$ in $2$ points. Both are smooth weak Del Pezzo surfaces, of degrees $8$ and $6$, respectively. The component $\overline{F^+(Y_{2,2})}$ can be obtained as a toric contraction of two $(-1)$-curves  in $S$ (intersecting a common $(-2)$-curve).
        \item[($B_{41}$)] The component $\overline{F^-(Y_{1,4})}$ is a Hirzebruch surface $\F_1$, and $\overline{F^+(Y_{1,4})}$ is singular weak Del Pezzo surface of degree $6$ that can be obtained as a toric contraction of two $(-2)$-curves  in $S$ (intersecting a common $(-1)$-curve). 
    \end{itemize}
    We conclude noting that, in both cases, $Y_{i,k}=\overline{F^-(Y_{i,k})}\cap\overline{F^+(Y_{i,k})}$ is a rational curve, whose points are fixed by the fundamental action $\mu_i\circ g^{k-1}$ (see Remark \ref{rem:fundaction}).
\end{remark}


\section{Proof of the Main Theorem}\label{sec:mainproof}



The main tool of the proof of our result is a control of the group of birational automorphisms of $X_e$ induced by the action of $W$ on $F$, i.e. the image of the homomorphism $b:W\to \Bir(X_e)$ introduced in Definition \ref{def:bsigma}. For simplicity, we will use also the contraction $p:X_e\to \P^2$ in order to identify $\Bir(X_e)\simeq\Bir(\P^2)$.

\begin{definition}\label{def:bsigma2}
The image of the Weyl group $W$ via the group homomorphism $W\to \Bir(X_e)\to \Bir(\P^2)$ will be denoted by $W_p$, and called the {\em Cremona representation of $W$} in the geometric quotient $\P^2$ of $F$.
\end{definition}

\begin{proposition}\label{prop:CremonaWeyl}
    There exists a system of homogeneous coordinates $(x:y:z)$ in $\P^2$ such that
    \begin{itemize}[leftmargin=4mm]
    \item the images of $r_1,r_2\in W$ into $\Bir(\P^2)$ are given by
        \[\begin{cases}
        r_1 : \P^2\longrightarrow\P^2, & (x : y : z)\longmapsto (x-y+z : z : y),\\
        r_2 : \P^2\dashrightarrow\,\P^2, & (x : y : z)\longmapsto (y^2:xy:-xz),
        \end{cases}
        \]
        \item the images in $\P^2$ of the boundary curves are described as in the table below (see Figure \ref{fig:boundary}).
\begin{table}[ht!!]
\centering
\begin{tabular}{|c|c|c|c|}
\hline
 $A_1$ (point) & $A_2$ (line) & $A_3$ (conic) & $A_4$ (line) \\\hline
$(1:0:0)$ & $x = 0$ & $xy-y^2+xz = 0$ & $x-y+z = 0$ \\
\hline\hline
$B_{12}$ (point)& $B_{23}$ (line)& $B_{34}$ (line) & $B_{41}$ (point) \\\hline
$(1:1:0)$ & $y = 0$ & $z = 0$ & $(0:0:1)$ \\\hline
\end{tabular}
\end{table}
    \end{itemize}
\end{proposition}

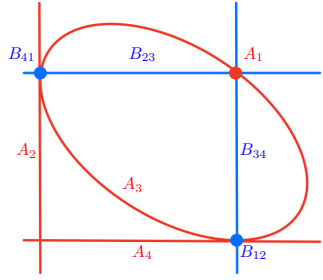
\begin{figure}[ht!!]
\centering
\resizebox{0.35\textwidth}{!}{%
\begin{tikzpicture}
     \node at (0,0) {\includegraphics[width=0.5\textwidth]{ 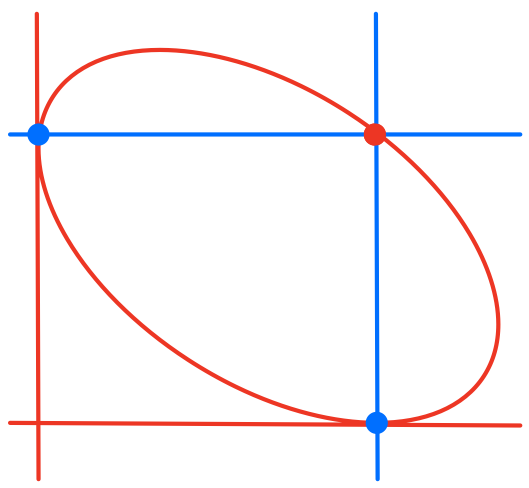}};
     \node at (1.7,-2.4)  {\color{blue}$B_{12}$};
     \node at (-0.6,1.7) {\color{blue}$B_{23}$};
     \node at (1.7,-0.3) {\color{blue}$B_{34}$};
     \node at (-3.1,1.7){\color{blue}$B_{41}$};
     \node at (1.7,1.7) {\color{red}$A_1$};
     \node at (-3,-0.3) {\color{red}$A_2$};
     \node at (-0.8,-1) {\color{red}$A_3$};
     \node at (-0.6,-2.4) {\color{red}$A_4$};
\end{tikzpicture}
}
\caption{The images in $\P^2$ of the boundary curves.}\label{fig:boundary}
\end{figure}
\begin{proof}
   The proof relies essentially in a direct computation of the maps $r_1$, $r_2$ as birational automorphisms of $\P^2$. 

In order to do so, we will describe first $r_i$, $i=1,2$, as a birational automorphisms of the open set $F_e=\exp(\fn)B\subset F$, which is isomorphic to $\fn$. We start by noting that an element $gB$ in $F_e$ can be written uniquely as the class modulo $B$ of (the homothety class of) a matrix:
\[g=\left(\begin{array}{rrrr}
1 & 0 & 0 & 0 \\
x_{1} & 1 & 0 & 0 \\
y & x_{2} & 1 & 0 \\
z & -x_{1} x_{2} + y & -x_{1} & 1
\end{array}\right), \quad x_1,x_2,y,z\in \C.\]
We will use $x_1,x_2,y,z$ as coordinates in $F_e\simeq \fn$, noting that these are different from the coordinates provided by the Cartan decomposition of $\fn$, but the action of $H$ on $F_e$ written in these coordinates is still the same: an element $h\in H$ sends $(x_1,x_2,y,z)$ to $(h^{\alpha_1}x_1,h^{\alpha_2}x_2,h^{\alpha_1+\alpha_2}y,h^{2\alpha_1+\alpha_2}z)$.
As already noted in Remark \ref{rem:geomP2}, the quotient map $F_e\dashrightarrow \P^2$ can be written as:
\[
(x_1,x_2,y,z)\in F_e\longmapsto (x_2:yx_1^{-1}:zx_1^{-2})\in\P^2.
\]
We use these as our homogeneous coordinates in $\P^2$, denoting them $(x:y:z)$. 

Now, for a general $gB\in F_e$, one can easily check that 
\[\begin{array}{l}
r_1g=\begin{pmatrix}
 1 & 0 & 0 & 0 \\
 {x_{1}^{-1}} & 1 & 0 & 0 \\
 {x_{1}^{-1}}z & x_{1}^{2} x_{2} - x_{1} y + z & 1 & 0 \\
  {x_{1}^{-1}}y & -x_{1} x_{2} + y & {-x_{1}^{-1}} & 1
 \end{pmatrix}
 \begin{pmatrix}
 {x_{1}^{-1}} & 1 & 0 & 0 \\
 0 & -x_{1} & 0 & 0 \\
 0 & 0 & {-x_{1}^{-1}} & 1 \\
 0 & 0 & 0 & x_{1}
 \end{pmatrix},\\[8pt]
r_2g=\begin{pmatrix}
 1 & 0 & 0 & 0 \\
 y & 1 & 0 & 0 \\
 -x_{1} & -{x_{2}^{-1}} & 1 & 0 \\
 z & -x_{1} + yx_{2}^{-1} & -y & 1
 \end{pmatrix}
 \begin{pmatrix}
 1 & 0 & 0 & 0 \\
 0 & x_{2}^{-1} & -1 & 0 \\
 0 & 0 & x_{2} & 0 \\
 0 & 0 & 0 & 1
 \end{pmatrix}.
\end{array}\]
In other words, in our chosen set of coordinates in $F_e$, the maps $r_1,r_2$ are given by
\[\begin{array}{l}r_1(x_1,x_2,y,z)=(x_{1}^{-1},x_{1}^{2} x_{2} - x_{1} y + z,{x_{1}^{-1}}z,{x_{1}^{-1}}y),\\[5pt]
r_2(x_1,x_2,y,z)=(y,-x_{2}^{-1},-x_{1},z).
\end{array}\]
Passing now to the quotient $\P^2$, the induced maps are then given by
\[\begin{array}{l}r_1(x:y:z)=(x - y + z:z:y),\\[5pt]
r_2(x:y:z)=(-x^{-1}:-y^{-1}:y^{-2}z)=(y^2:xy:-xz).
\end{array}\]

For the second part of the statement, note first that, by definition of our boundary curves, the images in $\P^2$ of $A_2$ and $B_{41}$ are the line $x=0$ and the point $(0:0:1)$, respectively. By applying $r_1$ to them, we see that $A_4$ is the line $x-y+z=0$ and $B_{12}$ the point $(1:1:0)$. 
Now applying $r_2=r_2^{-1}$ to $A_2$ and $B_{41}$, we get that $A_1$ is $(1:0:0)$ and $B_{23}$ is the line $y=0$. Finally, we may compute $A_3$ by applying $r_2$ to $A_4$, and $B_{34}$ as the image of $B_{23}$ by $r_1$.
\end{proof}

\begin{remark}\label{rem:r1r2}
    While the map $r_1\in W_p$ is a linear automorphism, the Cremona transformation $r_2\in W_p$ is the rational map induced by the linear system of conics of $\P^2$ passing through $A_1$ and tangent to the line $A_2$ at the point $B_{41}$.
\end{remark}

\begin{remark}\label{rem:Xr1}
We note that, by our description of the map $r_1$ above, it follows that the combinatorial quotient $X_{r_1}$ contracts to our base geometric quotient $\P^2$ and the map $p_{r_1}:X_{r_1}\to \P^2$ is a weighted blowup along the point $B_{12}$. Note also that the map $X_e\dashrightarrow X_{r_1}$ is not toric, since the weighted blowup at $B_{12}$ is performed with respect to the system of coordinates $((x-y+z)/y, z/y)$.  
\end{remark}

As an immediate consequence of Proposition \ref{prop:CremonaWeyl} we have the following result.
\begin{corollary}
    The Cremona representation $W_p\subset\Bir(\P^2)$ is isomorphic to the Weyl group $W\simeq D_4$. Moreover, $W_p$ acts effectively on the set of boundary divisors preserving their type. 
\end{corollary}

The next step for the proof of the Main Theorem is to find a resolution for the rational maps of $W_p$. 

The variety we are going to define is a birational modification of $\P^2$ via a sequence of suitable (weighted) blowups, which we will prove to be eventually isomorphic to the Chow quotient $X$ of $F$. 

\begin{definition}\label{def:X'}
    We denote by $X'$ the normal $\Q$-Cartier surface obtained upon $\P^2$ (in the usual coordinates $(x:y:z)$) by the following procedure: \begin{itemize}[leftmargin=4mm]
        \item given the maps $p:X_e\to \P^2$, $p_{r_1}:X_{r_1}\to \P^2$, we consider the fiber product
        $Z:=X_e\times_{\P^2}X_{r_1}$ defined by
        
        \[
\xymatrix{Z\ar[r]\ar[d]\ar[rd]^s&X_e\ar[d]^p\\X_{r_1}\ar[r]_{p_{r_1}}&\P^2}
        \]
        \item we then define $X'$ as the smooth blowup $t:X'\to Z$ of $Z$ along $s^{-1}(A_1)$.
    \end{itemize}
    We denote by $q=s\circ t:X'\to \P^2$ 
    the defining contraction.
\end{definition}

\begin{remark}\label{remark:X'}
    In other words, $X'$ is obtained by blowing up $\P^2$ along $A_1$, and  $B_{41}$, $B_{12}$ with certain weights. The construction of $X'$ is not affected by changing the order of these blowups.
    Alternatively, we may have defined:
    \[
X'\simeq\Bl_\cI\P^2=\Proj_{\P^2}\left(\bigoplus_{d\geq 0}\cI^d\right),\quad \cI=(x z^{2}, (x-y) y z, (x-y+z) y^{2})\cO_{\P^2}.
    \]
     The ideal sheaf $\cI$ is obtained as the intersection of the three ideal sheaves 
     \[(y,z)\cO_{\P^2},\quad (xz,y^2)\cO_{\P^2},\quad ((x-y+z)y,z^2)\cO_{\P^2}.\] 
\end{remark}

\begin{lemma}\label{lemma:p'r'=rp'}
    There exist unique automorphisms $r'_i:X'\to X',i=1,2$ such that 
    \begin{equation}\label{eq:p'r'=rp'}
    q\circ r'_i=r_i\circ q.\end{equation}
    \begin{proof}
        Consider first $r_2:\P^2\dashrightarrow\P^2$ and observe that $r_2=\phi\circ\rho_2$ where \[\rho_2:\P^2\dashrightarrow\P^2,\quad (x:y:z)\longmapsto(y^2:xy:xz)\] is a toric birational map and $\phi:(x:y:z)\in\P^2\mapsto(x:y:-z)\in\P^2$. First, $\rho_2$ lifts to a toric birational map $X_e\dashrightarrow X_e$, resolving the indeterminacy at $B_{41}$, and then to a toric automorphism $\rho''_2:\Bl_{A_1}(X_e)\to\Bl_{A_1}(X_e)$ that resolves the indeterminacy at (the inverse image of) $A_1$. Since $\phi$ fixes the points $A_1,B_{41}$ and the pencil of conics $\langle xz,y^2\rangle\subset|\cO_{\P^2}(2)|$, it also lifts to an automorphism $\phi'':\Bl_{A_1}(X_e)\to\Bl_{A_1}(X_e)$. Now the composition $r_2'':=\phi''\circ\rho''_2$ satisfies 
        \begin{equation}\label{eq:r''2}
            (p\circ t')\circ r''_2=r_2\circ(p\circ t')
        \end{equation} where $t':\Bl_{A_1}(X_e)\to X_e$ is the (smooth) blowup. Finally, note that $r_2$ is a local isomorphism in the fixed point $B_{12}$, therefore $r''_2$ extends to an automorphism $r'_2:X'\to X'$ via the natural 
        projection $f:X'\simeq\Bl_{A_1}(X_e)\times_{\P^2} Y\to\Bl_{A_1}(X_e)$. By Remark \ref{remark:X'} we have $q=p\circ t'\circ f$, thus (\ref{eq:r''2}) implies (\ref{eq:p'r'=rp'}) for $i=2$. 

        Now consider $r_1:\P^2\to\P^2$ and observe that, by Remark \ref{rem:Xr1}, it extends to an automorphism: $r_1'':Z\to Z$. Moreover, since $r_1$ 
        fixes the point $A_1$, then $r_1''$ extends to an automorphism $r_1':X'\to X'$, which satisfies (\ref{eq:p'r'=rp'}) for $i=1$.
    \end{proof}
\end{lemma}

\begin{corollary}\label{cor:autgp}
    There exists a unique group morphism \[b':w\in W\longmapsto b'_\sigma\in\Aut(X')\quad\textrm{such that}\quad b'_{r_i}=r'_i\quad\textrm{for}\,\,i=1,2.\] In particular, the Weyl group $W$ acts faithfully on $X'$.
\end{corollary}

\begin{theorem}
    The normalized Chow quotient $X$ is isomorphic to $X'$. In particular, $X$ is a $\Q$-Cartier rational surface with Picard number $4$ and only two singularities of type $\DA_1$.
\end{theorem}

    \begin{proof}
        Denote by $\pi'_{e}:X'\to X_e$ the composition $X'\to Z\to X_e$ of the contractions in Definition \ref{def:X'}. For every $\sigma\in W$ we define the birational morphism: \[\pi'_\sigma:=q_\sigma\circ \pi'_e\circ b'_{\sigma^{-1}}:X'\to X_\sigma.\]
    On the other hand, by Lemma \ref{lemma:p'r'=rp'}, it holds that \[\pi'_e\circ b'_\sigma=b_\sigma\circ\pi'_e,\quad\textrm{for all }\sigma\in W\] as birational maps. Summing up, this implies that \[\pi'_\sigma\circ{\pi'_e}^{-1}=(q_\sigma\circ\pi'_e\circ b'_{\sigma^{-1}})\circ{\pi'_e}^{-1}=q_\sigma\circ b_{\sigma^{-1}}=\jmath_\sigma.\] By Remark \ref{rem:invlim}, we have a birational morphism $\psi:X'\to X$. Since $X$ and $X'$ are normal, in order to conclude that $\psi$ is an isomorphism it suffices to show that it is finite. Suppose by contradiction that $\psi$ contracts an irreducible curve $C\subset X'$. Then every $\pi'_\sigma$ contracts $C$ but this is impossible since the exceptional loci of all the $\pi'_\sigma$ have empty intersection.
    \end{proof}


\section{Geometry of the Chow quotient}\label{sec:constructX}


We will consider the numbering the edges of the moment octagon presented in Figure \ref{fig:numbering}.
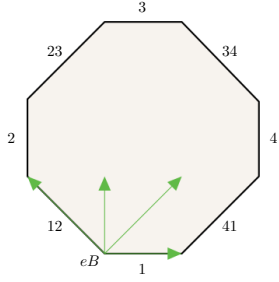
\begin{figure}[!ht]
\resizebox{0.3\textwidth}{!}{%
\begin{tikzpicture}
     \node[rotate = 180] at (0,0) {\includegraphics[width=0.46\textwidth]{ 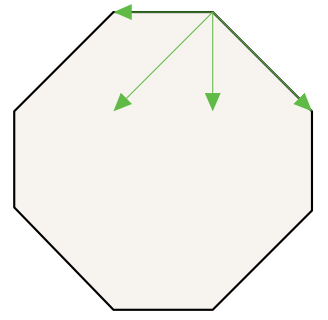}};
     \node at (-1.2,-2.8) {$eB$};
     %
     \node at (0,3) {$3$};
     \node at (0,-3) {$1$};
     \node at (-3,0) {$2$};
     \node at (3,0) {$4$};
     \node at (2,2) {$34$};
     \node at (-2,-2) {$12$};
     \node at (-2,2) {$23$};
     \node at (2,-2) {$41$};
\end{tikzpicture}
}
\caption{Numbering the edges of the octagon.\label{fig:numbering}}
\end{figure}

Taking into account the action of $W$ on the octagon $P$, this numbering is consistent with the representation of $W$ as the group of permutations of the set $\{1,2,3,4\}$, which identifies $r_1, r_2$ with the permutations $(2,4)$, $(1,2)(3,4)$, respectively. This numbering is consistent with the names given to the subdivisions of $P$ in Figure \ref{fig:subdiv}. We also recall the images of the eight boundary divisors via the contraction $q:X\simeq X'\to \P^2$, depicted in Figure \ref{fig:boundary}. 
\subsection{Intersection theory on $X$}\label{ssec:inttheory}

\begin{proposition}\label{prop:int}
With the above notation, the intersections of the boundary divisors in $X$ are the ones presented in Table \ref{tab:int}.

\begin{table}[ht!!]
\begin{tabular}{|c|c|c|c|c|c|c|c|c|}
\hline
&$A_1$&$A_2$&$A_3$&$A_4$&$B_{12}$&$B_{23}$&$B_{34}$&$B_{41}$\\
\hline
$A_1$&$-1$&$0$&$1$&$0$&$0$&$1$&$1$&$0$\\\hline
$A_2$&$0$&$-1$&$0$&$1$&$0$&$0$&$1$&$1$\\\hline
$A_3$&$1$&$0$&$-1$&$0$&$1$&$0$&$0$&$1$\\\hline
$A_4$&$0$&$1$&$0$&$-1$&$1$&$1$&$0$&$0$\\\hline
$B_{12}$&$0$&$0$&$1$&$1$&$-1/2$&$0$&$1/2$&$0$\\\hline
$B_{23}$&$1$&$0$&$0$&$1$&$0$&$-1/2$&$0$&$1/2$\\\hline
$B_{34}$&$1$&$1$&$0$&$0$&$1/2$&$0$&$-1/2$&$0$\\\hline
$B_{41}$&$0$&$1$&$1$&$0$&$0$&$1/2$&$0$&$-1/2$\\\hline
\end{tabular}
\vspace{0.5cm}\caption{Intersection numbers of boundary curves}\label{tab:int}
\vspace{-0.8cm}
\end{table}

    \begin{proof}
        Note first that if two boundary curves have nonzero intersection number then the intersection of the supports of their images in $\P^2$ must be nonempty. This already gives us all the zero values contained in the table. In order to conclude, by using the action of the group $W$ on $X$, it is enough to compute some particular intersection numbers:
        \begin{itemize}[leftmargin=4mm]
            \item $A_1^2=-1$ because $A_1$ is the exceptional divisor of a smooth blowup;
            \item $A_2\cdot A_4=1$ and $B_{23}\cdot A_4=1$ because, in both cases, these are intersection numbers of the strict transforms of two lines in $\P^2$, after blowing up points outside their intersection;
            \item the numbers $B_{41}^2=-1/2$ and $B_{41}\cdot B_{23}=1/2$ can be computed easily in $X_e$, by means of toric methods, noting that in Figure \ref{fig:quotfan}, the curves $B_{41}$ and $B_{23}$ corresponds to the rays generated by $(-1,-1)$ and $(1,-1)$, respectively.\qedhere 
        \end{itemize}
    \end{proof}
    
\end{proposition}

We know, from our constructive description of $X$, that its divisor class group is generated by the classes of our boundary divisors and, therefore, their numerical classes generate $\NU(X)=\Nu(X)$.

A straightforward linear algebra computation provides a basis for the space of numerical relations among the boundary divisors.

\begin{corollary}\label{cor:rels}
The subspace of numerical equivalence relations in the vector space $\R(A_1,\dots, B_{41})$ is generated by    
\[
\begin{array}{l}
A_1-A_2+A_3-A_4=B_{12}-B_{23}+B_{34}-B_{41}=0\\[3pt]
A_1-A_4-B_{12}+B_{34}=A_1-A_2+B_{23}-B_{41}=0.\qedhere
\end{array}
\]
\end{corollary}
In particular, we may easily compute a cokernel map for the inclusion of the space of relations. One possible choice provides the following description.
\begin{corollary}\label{cor:class}
With the above notation, there is an isomorphism $\NU(X)\simeq \R^4$ that sends the classes of the boundary divisors $A_1,\dots, A_4,B_{12},\dots B_{41}$ respectively to the columns of the matrix
\[\begin{pmatrix}
    1&1&1&1&0&0&0&0\\
    0&0&0&0&1&1&1&1\\
    1&0&-1&0&0&-1&-1&0\\
    0&-1&0&1&-1&-1&0&0
\end{pmatrix}.
\]
\end{corollary}
Finally, another linear algebra computation provides the following fact.
\begin{corollary}\label{cor:intform}
With the above notation, the bilinear form in $\R^4$ induced by the intersection product in $\NU(X)=\Nu(X)$ is given by the matrix
\[
\dfrac{1}{2}\begin{pmatrix}
    0&1&0&0\\1&-1&-1&-1\\0&-1&-2&0\\0&-1&0&-2
\end{pmatrix}.
\]
\end{corollary}

\subsection{Anticanonical embedding of $X$}\label{ssec:antican}
From now on, we will denote the numerical class of a divisor $D$ in $X$ by the same symbol $D$. Recall also the construction of the contraction $q:X\simeq X'\to\P^2$ in Definition \ref{def:X'}.

The standard ramification formula for (weighted) blowups provides
\begin{equation} \label{eq:antK}
\begin{aligned}
    -K_X&=q^*(-K_{\P^2})-2B_{41}-2B_{12}-A_1=\\
        &=2A_2+4B_{41}+A_4+2B_{12}-2B_{41}-2B_{12}-A_1=\\
        &=A_2+A_3-(A_1-A_2+A_3-A_4)+2B_{41}=\\
        &=A_2+A_3+2B_{41}.
\end{aligned}
\end{equation}
Note that the last equality follows by Corollary \ref{cor:rels}. Note also that, since $2B_{41}$ is Cartier, it follows that $-K_X$ is Cartier.

By means of Proposition \ref{prop:int}, one may directly check that the anticanonical degree of every boundary curve is $1$, a property that follows also by our constructive description of $X$. In particular, we may also write
\begin{equation} \label{eq:antK2}
-K_X=\dfrac{1}{2}(A_1+A_2+A_3+A_4+B_{12}+B_{23}+B_{34}+B_{41}),
\end{equation}
since the right hand side of the equality has degree $1$ with each boundary curve.
We may now prove the following fact.
\begin{corollary}\label{cor:nefantiK}
The variety $X$ is a singular Del Pezzo surface of degree $4$.
\end{corollary}

 \begin{proof}
We claim first that $-K_X$ is nef. If this were not the case, it would have negative intersection with a curve $C$, and then (\ref{eq:antK2}) would imply that $C$ is a boundary curve, contradicting that boundary curves have anticanonical degree one.

We now conclude the proof by using Nakai--Moishezon Theorem, and noting that, again by Proposition \ref{prop:int}, one gets $(-K_X)^2=4$.
 \end{proof}

Furthermore, the anticanonical linear system is very ample and embeds $X$ into $\P^4$ as a complete intersection of two quadrics. This property, known for general Del Pezzo's of degree $4$, can be checked explicitly in our case. In the next statement we will denote by
\[p_1:=x-y+z,\quad p_2:=xy-y^2+xz\]
the homogeneous polynomials defining $A_4$ and $A_3$, respectively.

\begin{proposition}\label{prop:anticanembed}
    The anticanonical linear system on $X$ is globally generated and its corresponding morphism $\varphi:X\to\P^4$ factors through $q:X\to\P^2$ via \[\varphi':\P^2 \dashrightarrow \P^4,\quad(x:y:z)\longmapsto(xyp_1:xz^2:xp_2:yp_2:
zp_2).\] 
\end{proposition}

\begin{proof}
First of all, recall that the map $q:X\to \P^2$ blows up $A_1=(1:0:0)$ smoothly and $B_{12}=(1:1:0), B_{41}=(0:0:1)$ with suitable weights.  
Via this map we may identify the anticanonical linear system of $X$ with the linear system of cubics in $\P^2$ passing through $A_1=(1:0:0)$ and passing through $B_{12}=(1:1:0)$, $B_{41}=(0:0:1)$ tangent to the lines $A_4:x-y+z=0$, $A_2:x=0$, respectively. It is then a straightforward computation to check that $xyp_1,xz^2,xp_2,yp_2,
zp_2$ form a basis of that linear system. 

We denote by $\varphi':\P^2 \dashrightarrow \P^4$ the corresponding rational map, so that we get a factorization $\varphi=\varphi'\circ q$. It remains to show that $\varphi$ is globally generated, i.e. that $\varphi$ resolves the indeterminacies of $\varphi'$.

Note that the indeterminacy locus of $\varphi'$ is supported on the three points $A_1,B_{12}, B_{41}\in \P^2$. More precisely, denoting by $\cJ$ the ideal sheaf of the indeterminacy locus, its localizations on these three points are 
\[\cJ_{A_1}=(y/x,z/x)\cO_{A_1},\quad 
\cJ_{B_{41}}=(x/z,(y/z)^2)\cO_{B_{41}},\quad
\cJ_{B_{12}}=(p_1/y,(z/y)^2)\cO_{B_{12}}.
\] 
Therefore, such indeterminacies are precisely resolved by the blowup $q:X\to\P^2$, as shown in Remark \ref{remark:X'}, so the map $\varphi$ is a morphism. 
\end{proof}

\begin{remark}\label{rem:anticanembed}
   From this description of $\varphi$, we immediately infer that it is an embedding in $X\setminus (A_1\cup A_3\cup B_{12}\cup B_{41})$. In fact, since $q(A_1\cup A_3\cup B_{12}\cup B_{41})=A_3$, it is enough to prove that $\varphi'$ is an embedding in $\P^2\setminus A_3$. Since $A_3$ is the set of zeroes of the homogeneous polynomial $p_2$, we conclude by noting that $\varphi'$ composed with the linear projection $\varphi(X)\subset\P^4\dashrightarrow\P^2$ sending $(x_0:\cdots:x_4)\longmapsto(x_2:x_3:x_4)$ is the identity of $\P^2$. 
\end{remark}

\begin{remark}\label{rem:anticanaction}
The action of $W$ on $X$ extends to an action on the projective space $\P^4=\P(\HH^0(X,-K_X)^\vee)$. We may use Proposition \ref{prop:CremonaWeyl} to compute the images via $r_1^*,r_2^*$ of the chosen generators of $\HH^0(X,-K_X)$. In this way we compute the images of $r_1$ and $r_2$ in $\P^4$ written in the coordinates of our choice (Proposition \ref{prop:anticanembed}), which generate the image of $W$ into $\PGL(5)$. They are the projectivities associated to the matrices
\[
R_1=\begin{pmatrix}-1 & 1 & 1 & 0 & 0\\
0 & 1 & 0 & 1 & -1\\
0 & 0 & 1 & -1 & 1\\
0 & 0 & 0 & 0 & 1\\
0 & 0 & 0 & 1 & 0\end{pmatrix},\quad 
R_2=\begin{pmatrix}0 & 0 & 0 & -1 & 0\\
0 & 1 & 0 & 0 & 0\\
0 & -1 & 0 & -1 & 1\\
-1 & 0 & 0 & 0 & 0\\
-1 & 1 & 1 & 0 & 0\end{pmatrix}.
\]
\end{remark}

\begin{proposition}\label{prop:lines}
    The image via $\varphi$ of the eight boundary curves $A_i$, $B_{ij}$ are eight lines in $\P^4$. Each line of type $A$ (resp. $B$) meets exactly another line of type $A$ (resp. $B$) and two lines of type $B$ (resp. $A$).
\end{proposition}

\begin{proof}
Proposition \ref{prop:anticanembed} allows us to compute the images of the boundary curves different from $A_1$, $B_{41}$, $B_{12}$, showing that they are lines. The three remaining curves can then be computed, for instance, as $R_2(A_2)$, $R_2(B_{23})$, $R_1(B_{41})$. Once these lines are obtained, it is straightforward to describes their intersections.
The following table contains the final outcome of the above computations.
\begingroup
\setlength{\tabcolsep}{1pt} 
\renewcommand{\arraystretch}{1.3} 
\begin{table}[ht!!]
\resizebox{\textwidth}{!}{%
\begin{tabular}{|C||C|C|C|C|C|C|C|C|}
     \hline
     \varphi&A_1&A_2&A_3&A_4&B_{12}&B_{23}&B_{34}&B_{41}  \\\hline\hline
     A_1    &&&[e_0]&&&[e_{ 2 }]&[e_0\!+\!e_{ 2 }]&  \\\hline
     A_2    &&&&[e_3\!+\!e_4]&&&[e_3]&[e_4]  \\\hline
     A_3    &[e_0]&&&&[e_0\!+\!e_{ 1 }]&&&[e_{ 1 }]  \\\hline
     A_4    &&[e_3\!+\!e_4]&&&[e_{ 1 }\!-\!e_{ 2 }\!-\!e_3]&[e_{ 1 }\!-\!e_{ 2 }\!+\!e_4]&&  \\\hline
     B_{12} &&&[e_0\!+\!e_{ 1 }]&[e_{ 1 }\!-\!e_{ 2 }\!-\!e_3]&&&[e_0\!+\!e_{ 2 }\!+\!e_3]&  \\\hline
     B_{23} &[e_{ 2 }]&&&[e_{ 1 }\!-\!e_{ 2 }\!+\!e_4]&&&&[e_{ 1 }\!+\!e_4] \\\hline
     B_{34} &[e_0\!+\!e_{ 2 }]&[e_3]&&&[e_0\!+\!e_{ 2 }\!+\!e_3]&&&  \\\hline
     B_{41} &&[e_4]&[e_{ 1 }]&&&[e_{ 1 }\!+\!e_4]&&  \\\hline
\end{tabular}}
\end{table}
\endgroup
\end{proof}

\begin{corollary}\label{cor:anticanembed}
    The anticanonical morphism $\varphi:X\to \P^4$ is an embedding.
\end{corollary}

\begin{proof}
    We have already discussed in Remark \ref{rem:anticanembed} that $\varphi_{|X\setminus(A_1\cup A_3 \cup B_{12}\cup B_{41})}$ is an embedding. Since $\varphi$ is $W$-equivariant (Remark \ref{rem:anticanaction}), it must be an embedding at $\sigma(X\setminus(A_1\cup A_3 \cup B_{12}\cup B_{41}))$, for every $\sigma\in W$, hence it is locally an embedding at every point of $X$. Then it remains to show that $\varphi$ is injective, and this follows from the fact that it is an embedding on every $\sigma(X\setminus(A_1\cup A_3 \cup B_{12}\cup B_{41}))$ and that the set of boundary curves maps injectively to $\P^4$ by   Proposition \ref{prop:8bounddiv}.
\end{proof}

\begin{remark}\label{rem:quadrics} 
Another straightforward computation tells us what is the linear system of quadrics in $\P^4$ containing $\varphi(X)$. Denoting our chosen homogeneous coordinates in $\P^4$ by $(X_0:\dots:X_4)$, our linear system of quadrics is generated by
\[Q_1:(X_0-X_2)X_4+X_1(X_2-X_3)=0,\quad Q_2:(X_0-X_2)X_3+X_2(X_4-X_1)=0.\]
These are two quadric cones, whose vertices are, respectively, 
\[[e_1+e_4]=\varphi(B_{23}\cap B_{41})\,\mbox{ and }\,[e_0+e_{2}+e_3]=\varphi(B_{12}\cap B_{34}),\] 
i.e. the images of the two singular points of $X$. Note that the line in $\P^4$ joining these two points is contained in the two quadrics, being the strict transform of the line in $\P^2$ of equation $x-y=0$.

Now, since $\deg\varphi(X)=4$, we may conclude that
$$
X\simeq \varphi(X)=Q_1\cap Q_2.
$$
In particular, $X$ (which, is a singular Del Pezzo surface of degree $4$ by Proposition \ref{cor:nefantiK}) can be obtained as a degeneration of a smooth Del Pezzo surface, which is the complete intersection of two general quadrics in $\P^4$.
\end{remark}


\subsection{The Mori cone of $X$ and its contractions}\label{sec:birgeomX}

It is well known that the general Del Pezzo surface of degree $4$ has $(\Z/2\Z)^4$ as automorphism group and that its Mori cone is the cone generated by the classes of its sixteen $(-1)$-curves, 
whose group of symmetries is the Weyl group of type $\DD_5$. Some automorphism groups of singular Del Pezzo surfaces are known, particularly in the case in which they have precisely one $\DA_1$ singularity (cf. \cite{HKT13}, \cite{Ho96}). 

In our case, we may describe the Mori cone of $X$ as follows.

\begin{proposition}\label{prop:Moricone}
    The Mori cone of $X$ is generated by the classes of the boundary curves. It is the cone over a square antiprism, which has $8$ vertices and $16$ edges.
\end{proposition}

\begin{proof}
We denote by $N\subset\NU(X)=\Nu(X)$ the cone generated by the classes of the boundary curves. Using the coordinates in $\NU(X)\simeq\R^4$ provided in Corollary \ref{cor:class} -- that we denote here by $(y_1,y_2,y_3,y_4)$ -- we note that the convex hull of the classes of boundary curves is contained in the hyperplane of equation $y_1+y_2=1$. So this polyhedron is equivalent to its projection via $(y_1,y_2,y_3,y_4)\longmapsto (y_2,y_3,y_4)$,  that we have represented in Figure \ref{fig:cones}. 

In particular, we see that the convex hull of the classes of boundary curves is a square antiprism, that has $8$ triangular faces (divided in $2$ $W$-orbits of $4$ elements) and $2$ ($W$-invariant) square faces, represented in Table \ref{tab:antiprism} below.

\begin{table}[ht!!]
    \centering
    \renewcommand{\arraystretch}{2} 
    \begin{tabular}{|c|c|}
        \hline
        \makecell{$\conv(A_1, B_{41}, B_{12}),\,\, \conv(A_2, B_{12}, B_{23}),$ \\ $\conv(A_3, B_{23}, B_{34}),\,\, \conv(A_4, B_{34}, B_{41})$} & 
        $\conv(A_1, A_2, A_3, A_4)$ \\
        \hline
        \makecell{$\conv(B_{12}, A_1, A_2),\,\, \conv(B_{23}, A_2, A_3),$ \\ $\conv(B_{34}, A_3, A_4), \,\,\conv(B_{41}, A_4, A_1)$} & 
        $\conv(B_{12}, B_{23}, B_{34}, B_{41})$ \\
        \hline
    \end{tabular}
    \vspace{0.3cm}
    \caption{\centering Faces of the square antiprism of boundary curves, grouped in $W$-orbits}\label{tab:antiprism}
\end{table}
\vspace{-0.3cm}

Now consider the corresponding facets of the cone $N$, that we denote by $\langle A_1,B_{41},B_{12}\rangle,\dots, \langle B_{12},B_{23},B_{34},B_{41}\rangle$. In order to show that $N=\cNE{X}$, we will prove that all the facets of $N$ are facets of $\cNE{X}$, by explicitly describing their associated contractions. Taking into account the action of $W$, it is enough to describe these contractions for one element in each $W$-orbit.  

By construction, $\langle A_1,B_{41},B_{12}\rangle$, corresponds to the contraction $q:X\to \P^2$, composition of $\pi_e:X\to X_e$ with the contraction $p:X_e\to\P^2$. In the second orbit we describe $\langle B_{12},A_1,A_2\rangle$, that corresponds to the composition of $\pi_e:X\to X_e$ with the contraction $p':X_e\to Q^2$ (see Remark \ref{rem:combiquot}).

In order to describe the facet $\langle B_{12},B_{23},B_{34},B_{41}\rangle$, we note that the contraction of $B_{12}$, $B_{41}$ sends $X$ to the blowup $\widetilde{P}_1$ of $\P^2$ along $A_1$, that has a contraction onto $\P^1$; since $B_{23},B_{34}\subset\P^2$ are lines passing by $A_1$, it follows that their strict transforms into $\widetilde{P}_1$ are fibers of the contraction to $\P^1$. In other words, the face $\langle B_{12},B_{23},B_{34},B_{41}\rangle$ corresponds to a contraction of $X$ to $\P^1$.

Finally, in the case of the face $\langle A_1,A_2,A_3,A_4\rangle$, we note that the contraction of $A_1$ in $X$ is a conic bundle over $\P^1$. It is indeed the family of conics in $\P^2$ passing by $B_{12}$ and $B_{41}$ tangent to $A_4,A_2$, respectively. Since $A_2+A_4$ and $A_3$ are two conics of this family, it follows that this face corresponds to another contraction of $X$ to $\P^1$. This concludes the proof.
\end{proof}

\begin{remark}\label{rem:Moricone}
    Instead of referring to a explicit description of the contraction of $X$, one may conclude the proof of the above statement by showing nef divisors vanishing on the facets of $N$. This can be done with the following list of divisors, written in correspondence with the faces listed in Table \ref{tab:antiprism}.
    \begin{table}[ht!!]
    \centering
    \renewcommand{\arraystretch}{2} 
    \begin{tabular}{|c|c|}
        \hline
        \makecell{$A_1+A_3+2B_{41}+2B_{12},\,\, A_2+A_4+2B_{12}+2B_{23},$ \\ $A_1+A_3+2B_{23}+2B_{34},\,\, A_2+A_4+2B_{34}+2B_{41}$} & 
        $B_{12}+B_{23}+B_{34}+B_{41}$ \\
        \hline
        \makecell{$ A_1+A_2+B_{23}+B_{41},\,\, A_2+A_3+B_{12}+B_{34},$ \\ $A_3+A_4+B_{23}+B_{41}, \,\,A_4+A_1+B_{12}+B_{34}$} & 
        $A_1+A_2+A_3+A_4$\\
        \hline
    \end{tabular}
\end{table}

    Note that the above divisors are nonnegative linear combinations of the boundary divisors; therefore, as discussed in the Proof of Corollary \ref{cor:nefantiK}, it is enough to check that the intersection of these divisors with the boundary divisors are nonnegative, for which it suffices to use the intersection matrix presented in Proposition \ref{prop:int}. Moreover, this computation also shows that the above divisors support the claimed facets of $N$. The advantage of this approach is that the classes of the divisors above provide a set of generators of the extremal rays of $\Nef(X)$. We have represented (an affine section of) the Nef cone of $X$ in Figure \ref{fig:cones}.
\end{remark}

\begin{figure}[ht!!]
    \centering
    \includegraphics[width=0.4\textwidth]{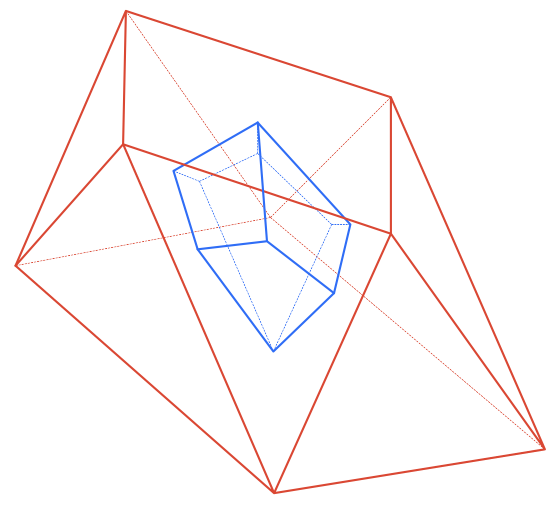}
    \caption{\centering Affine sections of the $\cNE{X}$ and $\Nef(X)$, represented in red and blue, respectively}
    \label{fig:cones}
\end{figure}
We conclude the section by proving that the only automorphisms of $X$ come from elements of the Weyl group $W$.
\begin{corollary}\label{cor:autgp2}
    The automorphism group of $X$ is isomorphic to $W$.
\end{corollary}

\begin{proof}
    Note first that we have a natural action of $W$ on $X$, which is faithful by Corollary \ref{cor:autgp}. We will show that the homomorphism $W\to \Aut(X)$ is surjective.

Given an automorphism $g\in\Aut(X)$, we consider the induced linear automorphism of $\NU(X)$, which stabilizes the effective cone $\cNE{X}$, sending extremal rays to extremal rays.

By Proposition \ref{prop:Moricone}, the extremal rays of $\cNE{X}$ are generated by the classes of the boundary curves and, by intersection-theoretical reasons, $g$ must preserve the type ($A$ or $B$) of the rays. It then follows that $g$ induces an automorphism of the square antiprism stabilizing its two bases. We conclude by noting that there exist only $8$ symmetries of the antiprism satisfying this property.
\end{proof}


\bibliographystyle{plain}
\bibliography{bibliomin}

\begin{thebibliography}{10}

\bibitem{Ale02}
Valery Alexeev.
\newblock Complete moduli in the presence of semiabelian group action.
\newblock {\em Ann. of Math. (2)}, 155(3):611--708, 2002.

\bibitem{AB04}
Valery Alexeev and Michel Brion.
\newblock Stable spherical varieties and their moduli.
\newblock {\em IMRP Int. Math. Res. Pap.}, pages Art. ID 46293, 57, 2006.

\bibitem{BHK}
Hendrik B{\"a}ker, J{\"u}rgen Hausen, and Simon Keicher.
\newblock {On Chow quotients of torus actions}.
\newblock {\em Michigan Mathematical Journal}, 64(3):451 -- 473, 2015.

\bibitem{BOS2}
Lorenzo Barban, Gianluca Occhetta, and Luis~E. Sol\'a~Conde.
\newblock Quotients of flag varieties and their birational geometry.
\newblock Preprint ArXiv:2509.05063, 2025.

\bibitem{BCHM}
Caucher Birkar, Paolo Cascini, Christopher~D. Hacon, and James McKernan.
\newblock Existence of minimal models for varieties of log general type.
\newblock {\em J. Amer. Math. Soc.}, 23(2):405--468, 2010.

\bibitem{BWW}
Jaros{\l}aw Buczy{\'n}ski, Jaros{\l}aw~A. Wi{\'s}niewski, and Andrzej Weber.
\newblock Algebraic torus actions on contact manifolds.
\newblock {\em J. Differential Geom.}, 121(2):227--289, 2022.

\bibitem{CK00}
James~B. Carrell and Alexandre Kurth.
\newblock Normality of {T}orus {O}rbit {C}losures in ${G}/{P}$.
\newblock {\em Journal of Algebra}, 233(1):122--134, 2000.

\bibitem{HKT13}
Brendan Hassett, Andrew Kresch, and Yuri Tschinkel.
\newblock On the moduli of degree 4 del {P}ezzo surfaces.
\newblock In {\em Development of moduli theory---{K}yoto 2013}, volume~69 of
  {\em Adv. Stud. Pure Math.}, pages 349--386. Math. Soc. Japan, Tokyo, 2016.

\bibitem{Ho96}
Toshio Hosoh.
\newblock Automorphism groups of quartic del {P}ezzo surfaces.
\newblock {\em J. Algebra}, 185(2):374--389, 1996.

\bibitem{Hu92}
Yi~Hu.
\newblock The geometry and topology of quotient varieties of torus actions.
\newblock {\em Duke Math. J.}, 68(1):151--184, 10 1992.

\bibitem{IVERSEN}
Birger Iversen.
\newblock A fixed point formula for action of tori on algebraic varieties.
\newblock {\em Inv. Math.}, 16(3):229--236, 1972.

\bibitem{Kap}
Mikhail~M. Kapranov.
\newblock Chow quotients of {G}rassmannians. {I}.
\newblock In {\em I. {M}. {G}el'fand {S}eminar}, volume~16 of {\em Adv. Soviet
  Math.}, pages 29--110. Amer. Math. Soc., Providence, RI, 1993.

\bibitem{KSZ}
Mikhail~M. Kapranov, Bernd Sturmfels, and Andrei~V. Zelevinsky.
\newblock Quotients of toric varieties.
\newblock {\em Math. Ann.}, 290(4):643--655, 1991.

\bibitem{KKN}
Alexander~M. Kasprzyk, Maximilian Kreuzer, and Benjamin Nill.
\newblock On the combinatorial classification of toric log del {P}ezzo
  surfaces.
\newblock {\em LMS J. Comput. Math.}, 13:33--46, 2010.

\bibitem{KKLV}
Friedrich Knop, Hanspeter Kraft, Domingo Luna, and Thierry Vust.
\newblock Local properties of algebraic group actions.
\newblock In {\em Algebraische {T}ransformationsgruppen und
  {I}nvariantentheorie}, volume~13 of {\em DMV Sem.}, pages 63--75.
  Birkh\"{a}user, Basel, 1989.

\bibitem{KKV}
Friedrich Knop, Hanspeter Kraft, and Thierry Vust.
\newblock {\em The Picard group of a $G$-variety}, pages 77--87.
\newblock Algebraische Transformationsgruppen und Invariantentheorie,
  Birkh{\"a}user Basel, Basel, 1989.

\bibitem{MFK}
David Mumford, John Fogarty, and Frances Kirwan.
\newblock {\em Geometric invariant theory}, volume~34 of {\em Ergebnisse der
  Mathematik und ihrer Grenzgebiete (2) [Results in Mathematics and Related
  Areas (2)]}.
\newblock Springer-Verlag, Berlin, third edition, 1994.

\bibitem{WORS5}
Gianluca Occhetta, Eleonora~A. Romano, Luis E.~Sol\'a Conde, and Jaros{\l}aw~A.
  Wi\'sniewski.
\newblock Rational homogeneous spaces as geometric realizations of birational
  transformations.
\newblock {\em Rend. Circ. Mat. Palermo, II Ser}, 2023.

\bibitem{OSS}
Christian Okonek, Michael Schneider, and Heinz Spindler.
\newblock {\em Vector bundles on complex projective spaces}.
\newblock Progress in Mathematics, 3. Birkh\"auser, Boston, Mass., 1980.

\end{thebibliography}


\end{document}